\documentclass{amsart}
\usepackage[T1]{fontenc}
\usepackage{lmodern}
\usepackage{amsmath,amssymb,enumerate,bm}
\usepackage{todonotes}
\usepackage{microtype}
\reversemarginpar
\usepackage{graphicx}
\usepackage{cite,subfigure}
\newcommand{\E}{\operatorname{E}}
\newcommand{\hE}{\widehat{\E}}
\newcommand{\Var}{\operatorname{Var}}
\newcommand{\N}{\mathbb{N}}
\newcommand{\R}{\mathbb{R}}

\newcommand{\bs}[1]{\bm{#1}}

\newcommand{\We}{\textrm{We}}
\newcommand{\Id}{\textrm{\bf{Id}}}

\newcommand{\U}{\bs{U}}
\newcommand{\F}{\bs{F}}
\newcommand{\Y}{\bs{Y}}
\newcommand{\Z}{\bs{Z}}
\newcommand{\ab}{\bs{a}}
\newcommand{\bb}{\bs{b}}
\newcommand{\x}{\bs{x}}

\newcommand{\bbf}{\bs{f}}
\newcommand{\W}{\bs{W}}
\newcommand{\X}{\bs{X}}
\newcommand{\dt}{\delta t}
\newcommand{\Dt}{\Delta t}
\newcommand{\mcX}{\mathcal{X}}
\newcommand{\mcY}{\mathcal{Y}}
\newcommand{\mcZ}{\mathcal{Z}}
\newcommand{\bX}{\bs{\mcX}}
\newcommand{\bY}{\bs{\mcY}}
\newcommand{\bZ}{\bs{\mcZ}}
\newcommand{\Q}{\mathcal{R}}
\newcommand{\Hc}{\mathcal{H}}

\newcommand{\I}{I}
\newcommand{\IDt}{\I^{\Dt}}

\newcommand{\os}{\phi}
\newcommand{\tRos}{\tilde{R}_{\os}}
\newcommand{\Ros}{{R}_{\os}}
\newcommand{\tSos}{\tilde{S}_{\os}}
\newcommand{\hRos}{\hat{R}_{\os}}
\newcommand{\hSos}{\hat{S}_{\os}}
\newcommand{\pe}{p_e}
\newcommand{\pos}{p_{\os}}
\newcommand{\eind}{j}
\newcommand{\Eind}{J}
\newcommand{\tind}{n}
\newcommand{\Tind}{N}
\newcommand{\fp}{\gamma}
\newcommand{\PO}{\mathcal{P}}
\newcommand{\tu}{\tilde{g}}
\newcommand{\tC}{\tilde{C}}
\providecommand{\id}{1\kern-0.25em{\rm l}}

\newtheorem{theorem}{Theorem}[section]
\usepackage{aliascnt}
\newaliascnt{definition}{theorem}
\newaliascnt{corollary}{theorem}
\newaliascnt{conjecture}{theorem}
\newaliascnt{remark}{theorem}
\newaliascnt{algorithm}{theorem}
\newaliascnt{example}{theorem}
\newaliascnt{lemma}{theorem}
\theoremstyle{plain}
\newtheorem{cor}[corollary]{Corollary}
\newtheorem{lem}[lemma]{Lemma}
\theoremstyle{definition}
\newtheorem{defi}[definition]{Definition}
\newtheorem{conjecture}[conjecture]{Conjecture}
\newtheorem{remark}[remark]{Remark}
\newtheorem{algorithm}[algorithm]{Algorithm}
\newtheorem{example}[example]{Example}
\aliascntresetthe{definition}
\aliascntresetthe{corollary}
\aliascntresetthe{conjecture}
\aliascntresetthe{remark}
\aliascntresetthe{algorithm}
\aliascntresetthe{example}
\aliascntresetthe{lemma}
\usepackage[capitalize,nosort]{cleveref}[2010/09/04]
\crefname{subsection}{Subsection}{Subsections}\Crefname{subsection}{Subsection}{Subsections}
\crefname{equation}{equation}{equations}\Crefname{equation}{Equation}{Equations}
\crefname{algorithm}{Algorithm}{Algorithms}\Crefname{algorithm}{Algorithm}{Algorithms}

\allowdisplaybreaks
\DeclareMathOperator*{\argmin}{{\rm argmin}}
\usepackage{scrtime}
\begin{document}
\title[Accelerated micro/macro Monte Carlo simulation of SDEs]{A micro/macro algorithm to accelerate Monte Carlo simulation of stochastic differential equations}
\thanks{This work was supported by the Research Council of the K.U.\ Leuven through grant OT/09/27 and by the Interuniversity Attraction Poles Programme of the Belgian Science Policy Office through grant IUAP/V/22. The scientific responsibility rests with its authors.}
\author{Kristian Debrabant}\address{Department of Mathematics and Computer Science, University of Southern Denmark, Campusvej 55, 5230 Odense M, Denmark ({\tt debrabant@imada.sdu.dk}). Part of this research was performed while KD being a Postdoctoral Research Fellow at the Department of Computer Science, K.\,U. Leuven.}
\author{Giovanni Samaey}\address{Department of Computer Science, K.\,U. Leuven, Celestijnenlaan 200A, 3001 Heverlee, Belgium ({\tt Giovanni.Samaey@cs.kuleuven.be}). GS is a Postdoctoral Fellow of the Research Foundation -- Flanders (FWO).}
\begin{abstract}
We present and analyze a micro/macro acceleration technique for the Monte Carlo simulation of stochastic differential equations (SDEs) in which there is a separation between the (fast) time-scale on which individual trajectories of the SDE need to be simulated and the (slow) time-scale on which we want to observe the (macroscopic) function of interest. The method performs short bursts of microscopic simulation using an ensemble of SDE realizations, after which the ensemble is restricted to a number of macroscopic state variables.  The resulting macroscopic state is then extrapolated forward in time and the ensemble is projected onto the extrapolated macroscopic state.  We provide a first analysis of its convergence in terms of extrapolation time step and number of macroscopic state variables. The effects of the different approximations on the resulting error are illustrated via numerical experiments.
\end{abstract}




\maketitle


\section{Introduction}
In many applications, one considers a process that is modeled as a stochastic  differential equation (SDE), while one is ultimately interested in the time evolution of the expectation of a certain function of the state, i.\,e., in weak approximation.  Consider, for instance, the micro/macro simulation of dilute solutions of polymers \cite{Laso:1993p10000}, which will be the motivating example in this paper.   Here, an SDE models the evolution of the configuration of an individual polymer driven by the flow field, and the function of interest is a non-Newtonian stress tensor (the expectation of a function of the polymer configuration). For this type of problem, one often resorts to Monte Carlo simulation \cite{caflisch98mca}, i.\,e., the simulation of a large ensemble of realizations of the SDE, combined with ensemble averaging to obtain an approximation of the quantity of interest at the desired moments in time. For concreteness, we introduce the SDE
\begin{equation}\label{eq:SDE}
d\X(t)=\ab\big(t,\X(t)\big)~dt+\bb\big(t,\X(t)\big)\star d\W(t),\quad t \in\I:=[t^0,T],\quad \X(t^0)=\X_0,
\end{equation}
in which $\ab:\I\times\R^d\to\R^d$ is the drift, $\bb:I\times\R^d\to\R^{d\times m}$ is the diffusion,
and $\W(t)$ is an $m$-dimensional Wiener process. The initial value $\X_0$ is independent of $\W$ and follows some known distribution with density $\varphi_0(\x)$. As usual, \eqref{eq:SDE} is an abbreviation of the integral form
\[
\X(t)=\X_0+\int_{t^0}^t\ab\big(s,\X(s)\big)~ds+\int_{t^0}^t\bb\big(s,\X(s)\big)\star d\W(s),\quad t \in\I.
\]
The integral with respect to  $\W$ can be interpreted, e.\,g., as an It\^{o} integral with $\star\; d\W(s) = d\W(s)$ or as a Stratonovich integral with $\star\; d\W(s) = \circ \;d\W(s)$.
The function of interest for the Monte Carlo simulation is defined as the expectation $\E$ of a function $\bbf\big(\X(t)\big)$,
\[
\bar{\bbf}(t)=\E\bbf\big(\X(t)\big).
\]

The numerical properties of Monte Carlo simulations have been analyzed extensively in the literature. We mention studies on the order of weak convergence of explicit \cite{milstein95nio,kloeden99nso,komori07wso,roessler07sor,debrabant09foe,debrabant10rkm} and implicit \cite{kloeden99nso,komori08wfo,debrabant09ddi,debrabant08bao,debrabant11bao} time discretizations of SDE \eqref{eq:SDE}, the investigation of stability  \cite{hernandez92aso,hernandez93cas,saito96sao,higham00msa,tocino05mss}, and techniques for variance reduction \cite{newton94vrf,giles08mmc,giles08imm}. For more references, we refer to \cite{kloeden99nso}.
Also for strong approximation, there has been a growing interest in the study of numerical methods for stiff SDEs \cite{kloeden99nso,milstein98bim,burrage01sar,tian01itm,burrage02pcm,tian02tss,milstein03nmf,burrage04isr,abdulle08src,abdulle08srm}.

In this paper, we present and analyze a \emph{micro/macro} acceleration technique for the Monte Carlo simulation of SDEs of the type \eqref{eq:SDE} in which there exists a time-scale separation between the (fast) time-scale on which individual trajectories of the SDE need to be simulated and the (slow) time-scale on which the function $\bar{\bbf}(t)$ evolves. The proposed method is motivated by the development of recent generic multiscale techniques, such as \emph{equation-free} \cite{Kevrekidis:2009p7484,KevrGearHymKevrRunTheo03} and \emph{heterogeneous multiscale} methods \cite{EEng03,E:2007p3747}. We use the simulation of a dilute polymer solution as an illustrative example. The \emph{microscopic} level is defined via an ensemble $\bX\equiv (\X_\eind)_{\eind=1}^\Eind$ of $\Eind$ realizations evolving according to \cref{eq:SDE}; the \emph{macroscopic} level will be defined by a set of $L$ \emph{macroscopic state variables} $\U\equiv (U_l)_{l=1}^L$, with $U_l(t)=\E u_l(\X(t))$, for some appropriately chosen functions $u_l$.
The method exploits a separation in time scales by combining short bursts of \emph{microscopic} simulation with the SDE \eqref{eq:SDE} with a \emph{macroscopic} extrapolation step, in which only the macroscopic state $\U$ is extrapolated forward in time.  One time step of the algorithm can  be written as follows: (1) microscopic \emph{simulation} of the ensemble using the SDE \eqref{eq:SDE}; (2) \emph{restriction}, i.\,e., extraction of (an estimate of) the macroscopic state (or macroscopic time derivative); (3) forward in time \emph{extrapolation} of the macroscopic state; and (4) \emph{matching} of the ensemble that was available at the end of the microscopic simulation with the extrapolated macroscopic state.
Remark that the resulting method is fully explicit as soon as the microscopic simulation is explicit, and that the method can readily be implemented as a  higher order method by an appropriate choice of the extrapolation.

The main contributions of the present paper are the following:
\begin{itemize}
\item From a numerical analysis viewpoint, we study convergence
 of the proposed micro/macro acceleration method in the absence of statistical error.  Specifically, we discuss how the error that is introduced during the matching depends on the number of macroscopic state variables, and how the deterministic error depends on the extrapolation time step. Additionally, we also comment on the effects of the extrapolation on the statistical error. Finally, we give a general convergence result.
\item From a practical viewpoint, we provide numerical results for a nontrivial test case, showing the interplay between the different sources of numerical error.  We illustrate the effects of the choice of macroscopic state variables, as well as the dependence of the numerical error on the chosen extrapolation strategy.
\end{itemize}
A stability analysis of the proposed method will be given in a separate publication.

The remainder of the paper is organized as follows. In
\cref{sec:preliminaries}, we first introduce the mathematical setting of the
paper. We discuss the necessary assumptions on the SDE \eqref{eq:SDE} for our
Monte Carlo setting, and introduce the illustrative example that will be used
for the numerical experiments.
In \cref{subsec:new-algo}, we
propose the algorithm that will be the focus of this paper.
\Cref{sec:proj} provides some results on the matching operator, whereas the
extrapolation operator is discussed in \cref{subsec:accel-extrap}. We provide
a general convergence result in \cref{sec:conv}. \Cref{sec:numerical} provides
numerical illustrations, which are chosen to illuminate the properties of the
proposed method. We conclude in \cref{sec:concl}, where we also outline some
directions for future research.

\section{Mathematical setting\label{sec:preliminaries}}

In this section, we introduce in detail the notations that we will use (\cref{sec:math-set}), as well as the illustrative example that will be considered throughout the paper (\cref{sec:model}).

\subsection{Notations\label{sec:math-set}}

We first define the appropriate function spaces. Let $C_P^r(\R^d, \R)$ denote the space of all $g \in C^r(\R^d,\R)$ fulfilling that there are constants $\tilde{C}>0$ and $\kappa>0$ such that $|\partial^i_{\x} g(\x)|\leq\tilde{C}(1+\|\x\|^\kappa)$ for any partial derivative of order $i \leq r$ and all $\x\in\R^d$. Further, let $g \in C_P^{q,r}(\I\times\R^d, \R)$ if $g(\cdot,\x) \in C^{q}(\I,\R)$, $g(t,\cdot) \in C^r(\R^d, \R)$ for all $t \in\I$ and $\x \in \R^d$, and $|\partial^i_{\x}g(s,\x)|\leq\tilde{C}(1+\|\x\|^\kappa)$ holds for $0\leq i\leq r$ uniformly with respect to $s\in[t^0,t]$ and for all $\x\in\R^d$ \cite{kloeden99nso,milstein95nio}.

We consider the SDE \eqref{eq:SDE}.
Besides the exact solution $\X(t)$ of SDE \eqref{eq:SDE} starting from $\X(t^0)=\X_0$, we also introduce the exact solution of an auxiliary initial value problem for the SDE: the solution of the SDE starting from an initial value $\Z$ at time $t$ will be denoted as $\X^{t,\Z}$. With this notation, we state the following definition:
\begin{defi}[Uniform weak continuity of the SDE]\label{def:weakcontinuitySDE}
Consider a class of random variables. The
SDE \eqref{eq:SDE} is called uniformly weakly continuous for this class if, for all $g \in C_P^{0,2}(I\times\R^d, \R)$, there exist constants $\Dt_0>0$ and $C$ such that
\begin{equation}
|\E g\big(s,\X^{t,\Z}(t+\Dt)\big)-\E g(s,\Z)|\leq C\Dt
\end{equation}
holds for all initial values $\Z$ in the considered class, $\Dt\in[0,\Dt_0]$, and all $t\in[t^0,T-\Dt]$, $s\in\I$.
\end{defi}

Next, we discretize \cref{eq:SDE} in time with step size $\dt$ and denote the numerical approximation $\Y^{k}=\Y(t^k)\approx \X(t^k)$ with $t^k=t^0+k\dt$. Leaving the concrete choice of discretization undecided for now, we introduce a short-hand notation for an abstract one step discretization scheme,
\begin{equation}\label{eq:sde_discr}
	\Y^{k+1}=\os(t^k,\Y^k;\dt), \qquad \Y^0=\X_0, \qquad k\ge 0.
\end{equation}
In the Monte Carlo setting, we are interested in weak approximation of the SDE \eqref{eq:SDE}. We define the weak order of consistency of the discretization $\os$ as follows (compare \cite{milstein95nio}):
\begin{defi}[Weak order of consistency of SDE discretization]
Assume that for all $g \in C_P^{0,2(\pos+1)}(I\times\R^d, \R)$ there exists a $C_g\in C_P^0(\R^d,\R)$ such that
\begin{equation*}
|\E g\big(s,\os(t,\Z;\dt)\big)-\E g\big(s,\X^{t,\Z}(t+\dt)\big)| \leq C_g(\Z) \,\dt^{\pos+1}
\end{equation*}
is valid for $\Z \in \R^d$, $s\in\I$, and $t$, \mbox{$t+\dt \in [t^0,T]$}. Then the one step method $\os$ is called weakly consistent of order $\pos$.
\end{defi}

Since we can only simulate a finite number of realizations of approximations of \eqref{eq:SDE} (via its discretization \eqref{eq:sde_discr}), we also need to approximate the expectation $\E$ by an empirical mean $\hE$. Using the ensemble $\bY\equiv (\Y_\eind)_{\eind=1}^\Eind$ of realizations of \eqref{eq:sde_discr}, the empirical mean $\hE$ is  defined as
\[
\hE \bbf(\bY)=\frac{1}{\Eind}\sum_{\eind=1}^\Eind \bbf(\Y_\eind).
\]
The numerical integration scheme for the ensemble $\bY$ will be denoted as
\begin{equation}\label{eq:sde_discr_ens}
	\bY^{k+1}=\os_{\bY}(t^k,\bY^k;\dt)=\left(\os(t^k,\bY^k_\eind;\dt)\right)_{\eind=1}^\Eind.
\end{equation}

The total error of a Monte Carlo simulation of \eqref{eq:SDE} consists of a \emph{deterministic} error due to the time discretization, and a \emph{statistical} error due to the finite number of realizations. In general, for a given tolerance, a Monte Carlo simulation will be most efficient if the statistical and deterministic error are balanced.  However, as for stiff systems of ODEs, one might encounter situations in which the time step $\dt$  cannot be increased because of stability problems. The acceleration method that we will propose in \cref{subsec:new-algo} will prove to be particularly useful to accelerate Monte Carlo simulations in such situations.

\begin{remark}[Probability density functions]\label{rem:FokkerPlanck}
One can equivalently describe the process \eqref{eq:SDE} via an advection-diffusion equation, also known as Fokker--Planck equation (see, e.\,g., \cite{risken1996fokker}), which in the It\^{o} case takes the form
\begin{equation}
	\partial_t\varphi=-\nabla \left(\ab\;\varphi\right)
	+ \frac{1}{2}\nabla\cdot\left[\nabla\cdot\left(\bb^T\bb\;\varphi\right)\right],
\end{equation}
and which describes the evolution of the probability density function $\varphi(t,\x)$ of $\X(t)$, starting from the initial density $\varphi(t^0,\x)=\varphi_0(\x)$.  The functional of interest then becomes
\begin{equation}\label{eq:intexp}
\bar{\bbf}(t)=\int \bbf(\x) \varphi(t,\x)d\x.
\end{equation}
For simulation purposes, however, the Monte Carlo algorithm is generally preferred, due to the possibly high number of dimensions of the Fokker--Planck equation.
\end{remark}
\begin{remark}[Spatial dimension] While the model \eqref{eq:SDE} can have arbitrary dimension, we will consider a one-dimensional version in the numerical illustrations for ease of visualization. In that case the state vector reduces to a scalar, $X(t)$. Whenever we consider the one-dimensional case, all bold typesetting in \cref{eq:SDE} will be removed.
\end{remark}

\subsection{A motivating model problem: FENE dumbbells\label{sec:model}}

To illustrate the behavior of the proposed numerical methods, we will consider the micro/macro simulation of the evolution of immersed polymers in a solvent. Here, one models the evolution of the configuration of a polymer ensemble via an SDE of the type \eqref{eq:SDE}, driven by the flow field, for each of the individual polymers. This results in a polymer configuration distribution at each spatial point. This microscopic model is coupled to a Navier--Stokes equation for the solvent, in which the effect of the immersed polymers is taken into account via a non-Newtonian stress tensor.  We refer to~\cite{Hulsen:1997p7027,Laso:1993p10000,le-bris-lelievre-09} for an introduction to the literature on this subject.

In this paper, we consider only the Monte Carlo simulation of the microscopic model, leaving the coupling with the Navier--Stokes equations for future work.  We eliminate the spatial dependence by considering the microscopic equations along the characteristics of the flow field, i.\,e., in a Lagrangian frame.  In general, a microscopic model describes an individual polymer as a series of beads, connected by nonlinear springs, resulting in a coupled system of SDEs for the position of each of the beads. In the simplest case, that we will also use as an illustrative example here, one represents the polymers as non-interacting dumbbells, connecting two beads by a spring that models intramolecular interaction. The state of the polymer chain is described by the end-to-end vector $\X(t)$ that connects both beads, and whose evolution is modelled using the non-dimensionalised SDE
\begin{equation}
d\X(t)=\left[\boldsymbol{\kappa}(t) \, \X(t)-\dfrac{1}{2\We}\F\big(\X(t)\big)\right]dt + \dfrac{1}{\sqrt{\We}}d\W(t), \label{eq:fene-3d}
\end{equation}
where $\boldsymbol{\kappa}(t)$ is the velocity gradient of the solvent, $\We$ is the Weissenberg number, and $\F$ is an entropic force, here considered to be finitely extensible nonlinearly elastic (FENE),
\begin{equation}\label{eq:springs}
	\F(\X)=\dfrac{\X}{1-\|\X\|^2/\fp},
\end{equation}
with $\fp$ a non-dimensional parameter that is related to the maximal polymer length. The resulting non-Newtonian stress tensor is given by the Kramers' expression,
\begin{equation}
\bs{\tau}_p(t) = \dfrac{\epsilon}{\We} \bigg( \E\Big( \X(t)\otimes \F\big(\X(t)\big)\Big)-\Id \bigg), \label{eq:kramers}
\end{equation}
in which $\epsilon$ represents the ratio of polymer and total viscosity; see \cite{le-bris-lelievre-09} for details and further references.
This model, which is of the type \eqref{eq:SDE}, takes into account Stokes drag (due to the solvent velocity field), intramolecular elastic forces, and Brownian motion (due to collisions with solvent molecules). The functional of interest in the Monte Carlo simulation is $\bar{\bbf}(t)=\bs{\tau}_p(t)$.

\Cref{eq:fene-3d,eq:springs} ensure that the length of the end-to-end vector, $\|\X\|$, cannot exceed the maximal value $\sqrt{\fp}$~\cite{jourdain-lelievre-02}. However, a naive explicit discretization scheme might yield polymer lengths beyond this maximal value. This can be avoided via an accept-reject strategy, see, e.\,g., \cite[Section 4.3.2]{Ott96}. Here, for each polymer, the state after each time step is rejected if the calculated polymer length exceeds $\sqrt{(1-\sqrt{\dt})\fp}$, and a new random number is tried until acceptance. To prevent the distribution of the approximation process to be heavily influenced, the microscopic time-step has to be chosen small enough.  Alternatively, one can use an implicit method \cite{Ott96}, which alleviates the time-step restriction. However, even for implicit SDE discretizations, the maximal time step is limited when coupling the Monte Carlo simulation with a discretization of the Navier--Stokes equations for the solvent. This is due to the fact that the coupling between the Monte Carlo and Navier--Stokes parts is, in most existing work, done explicitly in time, creating an additional stability constraint due to the coupling. (Some notable exceptions are given in \cite{Laso:2004p12587,Somasi:2000p12659}.)  With the algorithm that we propose, one would extrapolate both the Monte Carlo and the Navier--Stokes part of this coupled simulation simultaneously. This will be done in future work.
Here, we simply conclude that, for this model problem, the required time step for a stable SDE (or coupled) simulation may indeed be small compared to the time scale of the evolution of the stress.

As in, e.\,g., \cite{Keunings:1997p9982}, we will consider a one-dimensional version in the numerical illustrations; the stress tensor then reduces to a scalar $\tau_p(t)$. As time discretization we will use the explicit Euler--Maruyama scheme, combined with an accept-reject strategy.

\section{Micro/macro acceleration method\label{sec:accel}\label{subsec:new-algo}}

In this section, we describe the micro/macro acceleration algorithm that is the focus of the present paper.  As said above, the goal of the method is to be faster than a full microscopic simulation, while converging to a full microscopic simulation when the extrapolation time step vanishes (see \cref{sec:conv}). The algorithm combines short bursts of microscopic simulation of an ensemble $\bX$ of $J$ realizations of the SDE \eqref{eq:SDE} with a macroscopic extrapolation step. During this macroscopic extrapolation, only a set of $L$ macroscopic state variables $\U$ are extrapolated forward in time, and the microscopic ensemble then needs to be \emph{matched} onto the extrapolated macroscopic state.

\subsection{Description of algorithm}

Before giving a detailed description of the algorithm, let us first elaborate slightly on the macroscopic state variables, $\U=\left( U_l\right)_{l=1}^L$, which are defined as expectations of scalar functions $u_l$ of the state $\X$ and time $t$,
\begin{equation}\label{eq:state-vars}
U_l(t) = \E u_l\big(t,\X(t)\big).
\end{equation}
\begin{remark}[Choice of macroscopic state variables]
The choice of the functions $u_l$ is problem-dependent and will be specified  with the numerical illustrations for the examples considered in this text. For the exposition in this section, it may be helpful to think about the standard moments of the distribution in a one-dimensional setting, i.\,e., $u_l(t,x)=x^l$. Note that by allowing $u_l$ to depend directly on time $t$, centralized moments $u_1(t,x)=x$, $u_l(t,x)=(x-U_1(t))^l$ for $l\geq2$, for instance, can also be considered.
\end{remark}

We introduce two abstract operators to connect both levels of description: a
\emph{restriction} operator
\begin{equation}
\label{eq:intro_restriction}
\Q: \bY \mapsto \U=\Q(\bY),
\end{equation}
that maps a microscopic ensemble onto a macroscopic state, and a \emph{matching} operator
\begin{equation}
	\PO: \U,\bY^* \mapsto \bY=\PO(\U,\bY^*),
\end{equation}
which matches a given microscopic ensemble $\bY^*$ with an imposed macroscopic state $\U$.  Both operators will be discussed in detail in \cref{sec:match-restrict}.

We further introduce some notation.  Let ${\IDt} = \{t^0, t^1, \ldots, t^{\Tind}\}$ be a (macroscopic) time discretization  of the time interval $\I$, with $t^0 < t^1 < \ldots < t^{\Tind} =T$ and step
sizes $\Dt_{\tind} = t^{{\tind}+1}-t^{\tind}$ for ${\tind}=0,1, \ldots, \Tind-1$. 
We also introduce the discrete time instances $t^{{\tind},k}=t^{\tind}+k\dt$ that are defined on a microscopic grid, and, correspondingly, the discrete approximations $\U^{{\tind},k}\approx\U(t^{{\tind},k})$ and $\U^{{\tind}}\approx\U(t^{\tind})$, and analogously for $\bY$ and $\Y$. Clearly, $(\cdot)^{\tind}=(\cdot)^{{\tind},0}$.

One step of the micro/macro acceleration method then reads:
\begin{algorithm}[Micro/macro acceleration]\label{algo:accel}Given a microscopic state $\bY^{{\tind}}$ at time $t^\tind$, advance to a microscopic state $\bY^{{\tind}+1}$ at time $t^{\tind+1}$ via a three-step procedure:
	\begin{itemize}
	\item[(i)] \emph{Simulate} the microscopic system over $K$ time steps of size $\dt$, \[\bY^{{\tind},k}=\os_{\bY}(t^{{\tind},k-1},\bY^{{\tind},k-1};\dt), \qquad 1\le k \le K,\]
	and record the restrictions $\U^{{\tind},k}=\Q({\bY^{{\tind},k}})$, as well as an approximation of the function of interest $\hat{\bbf}^{{\tind},k}=\hE\bbf(\bY^{{\tind},k})$.
	\item[(ii)] \emph{Extrapolate} the macroscopic state  $\U$ from (some of) the time points $t^{i,k}$,$i=0,\ldots,n$, $k=1,\dots,K$, to a new macroscopic state $\U^{{\tind}+1}$ at time $t^{{\tind}+1}$,
	\begin{equation}\label{eq:extrap}
		\U^{{\tind}+1} = {\mathcal E}\left(\left(\U^{i,k}\right)_{i,k=0,0}^{n,K};(\Dt_i)_{i=0}^n,\dt\right).
	\end{equation}
	\item[(iii)] \emph{Match} the microscopic state $\bY^{{\tind},K}$ with the extrapolated macroscopic state $\bY^{{\tind}+1}=\PO(\bY^{{\tind},K},\U^{{\tind}+1})$.
\end{itemize}
	\end{algorithm}
Some basic requirements for the matching and restriction operators are given in \Cref{sec:match-restrict}. Specific choices and convergence properties for the algorithmic components are given in \Cref{sec:proj} (matching) and \Cref{subsec:accel-extrap} (extrapolation).

\begin{remark}[Closure approximation]
Due to the possibly high computational cost of Monte Carlo simulation, another route has been followed in the literature, in which one derives an approximate macroscopic model to describe the system; see, e.\,g., \cite{Herrchen:1997p8915,Hyon:2008p9897,Keunings:1997p9982,Lielens:1998p6790,Lielens:1999p9945,Sizaire:1999p9912} for derivations of macroscopic closures for FENE dumbbell models. The goal then is to obtain a closed system of $L$ evolution equations for the macroscopic state variables
$\U$, complemented with a constitutive equation, for the observable $\bar{\bbf}$ of interest as a function of these macroscopic state variables. 	
In~\cite{Ilg:2002p10825}, a \emph{quasi-equilibrium} approach is proposed, based on thermodynamical considerations; although the method has been formulated for the FENE dumbbell case, it is applicable to general SDEs of the type \eqref{eq:SDE}.  Several algorithms have been presented to simulate the evolution of the quasi-equilibrium model numerically \cite{SamLelLeg10,Wang:2008}.  Note that, in contrast with the method presented in the present paper, numerical closures introduce the modeling assumption that a closed model in terms of the macroscopic state variables exists. The micro/macro acceleration method presented here only uses these macroscopic state variables for computational purposes, and maintains convergence to the full microscopic dynamics (see \cref{sec:conv}).
\end{remark}
	
\subsection{Matching and restriction operators}\label{sec:match-restrict}

Next, we give some detail on the \emph{restriction} and \emph{matching} operators that connect the microscopic and macroscopic levels of description.  To obtain the macroscopic state from the microscopic ensemble, the restriction operator is defined, which can readily be obtained by replacing the expectation $\E$ by the empirical mean~$\hE$,
	\begin{equation}
	 	U_l(t) =\Q_l\big(\bY(t)\big) = \hE u_l\big(t,\bY(t)\big).
	\end{equation}

Due to the constraint $\Q(\bY^{{\tind}+1})=\U^{{\tind}+1}$, during the matching step, the elements of $\bY^{{\tind}+1}$ are in general not independent, and Monte Carlo error estimates are not straightforward to obtain. For example, the Central Limit Theorem is not directly applicable \cite{caflisch98mca}.
 Therefore, to discuss the matching operator
$\PO$, we first turn to the idealized operators $\overline{\PO}$ and $\overline{\Q}$ obtained in the limit $\Eind \to \infty$, eliminating statistical error. Rather than acting on ensembles of configurations, these operators are defined directly on the random variables. More precisely, the restriction operator $\overline{\Q}$ reduces a random variable $\Z$ to macroscopic state variables,
	\begin{equation}\label{eq:Rbar}
	 \overline{\Q}(\Z)=\big(\overline{\Q}_l(\Z)\big)_{l=1}^L \text{ with } \overline{\Q}_l(\Z) =U_l= \E u_l(\Z) \text{ for $l=1,\ldots,L$}.
	\end{equation}
	The matching operator $\overline{\PO}$  maps a random variable $\Z$ onto a new random variable $\Z^*$ that corresponds to a macroscopic state $\U^*$,
	\begin{equation}
		\Z^*=\overline{\PO}(\Z,\U^*)\text{ with }\overline{\Q}(\Z^*)=\U^*.
	\end{equation}
	When analyzing the matching, we will consider the idealized operators $\overline{\PO}$ and $\overline{\Q}$.

A priori, we allow for considerable freedom in the definition of the matching operator, requiring only a few properties to be satisfied for any reasonable matching. First, we impose a projection property:
	\begin{defi}[Projection property]\label{def:selfconsistency}
	A matching operator $\bar{\PO}$ associated to a restriction operator $\bar{\Q}$  satisfies the projection property if
	\[\U =\bar{\Q}(\Z) \Rightarrow \Z=\bar{\PO}\big(\Z,\U\big)=\bar{\PO}_{\U}(\Z)\]
	for any suitable random variable $\Z$, i.\,e., if $\bar{\PO}_{\U}$ is a projection operator for every $\U$.
	\end{defi}
	The projection property states that a random variable remains unaffected by projection if its macroscopic state is equal to the macroscopic state on which one wants to project.	We remark also that \cref{def:selfconsistency} implies $\bar{\PO}_{\U}^2=\bar{\PO}_{\U}$, which justifies the use of the term \emph{projection}.

	Next, we consider the number of macroscopic state variables $L$ to vary, and define a sequence of vectors of macroscopic state variables $\left(\U_{[L]}\right)_{L=1,2,\dots}$, such that $\U_{[L]}=(U_l)_{l=1}^{L}$, i.\,e., for increasing $L$, additional macroscopic state variables are added. The corresponding sequences of matching and restriction operators are denoted as $\left(\overline{\PO}_{[L]}\right)_{L=1,2,\ldots}$ and $\left(\overline{\Q}_{[L]}\right)_{L=1,2,\ldots}$, respectively.

	Using this notation, we are ready to formulate the definitions of continuity and consistency of the matching step:

	\begin{defi}[Continuity of matching]
	Consider a set of random variables and a set of sequences of macroscopic states. A sequence of matching operators $\left(\overline{\PO}_{[L]}\right)_{L=1,2,\dots}$ is called continuous for these sets if, for all $g \in C_P^{0,2}(I\times\R^d, \R)$, there exists a constant $C$, depending only on $g$, such that
	\begin{equation}\label{eq:continuity}
	 |\E g\big(s,\overline{\PO}_{[L]}(\Z,\U_{[L]}^*)\big)-\E g\big(s,\overline{\PO}_{[L]}(\Z,\U^+_{[L]})\big)|\leq C\|\U^*_{[L]}-\U^+_{[L]}\|
	\end{equation}
	holds for all $L\geq1$, all sequences $\left(\U^*_{[L]}\right)_{L=1,2,\dots}$, $\left(\U^+_{[L]}\right)_{L=1,2,\dots}$ of macroscopic states and all $\Z$ in the considered sets, and all $s\in\I$.
	\end{defi}
	Here and in the following, the norm $\|\cdot\|$ should be chosen such that it remains bounded for $L\to\infty$ for component-wise bounded sequences.

	\begin{defi}[Consistency of matching]\label{def:consistencyprojection}
	Consider a sequence of matching operators $\left(\overline{\PO}_{[L]}\right)_{L=1,2,\dots}$. This sequence is called consistent for a class of sequences of triples $\left(\Z^*_{[L]},\Z^+_{[L]},\U_{[L]}\right)_{L=1,2,\dots}$ of random variables $\Z^*_{[L]}$, $\Z^+_{[L]}$, and macroscopic states $(\U_{[L]})$, if for all $g\in C_P^{0,2}(I\times\R^d, \R)$ there exist constants $C_L$, with $C_L\to0$ for $L\to\infty$, and $L_0$ such that it holds
	\begin{equation}\label{eq:consistency}
	|\E g\big(s,\overline{\PO}_{[L]}(\Z^*_{[L]},\U_{[L]})\big)-\E g\big(s,\overline{\PO}_{[L]}(\Z^+_{[L]},\U_{[L]})\big)|
	\leq C_L
	|\E g(s,\Z^*_{[L]})-\E g(s,\Z^+_{[L]})|
	\end{equation}
	for all $L\geq L_0$ and all sequences of triples $\left(\Z^*_{[L]},\Z^+_{[L]},\U_{[L]}\right)_{L=1,2,\dots}$ in the considered class.
	\end{defi}
	The possible dependence of the random variables on $L$ is present since the random variables will be considered to have been generated using the micro/macro acceleration algorithm, and therefore depend on $L$.

	Whereas continuity measures (in a weak sense) the difference between the matching of a random variable with two different macroscopic states, the consistency measures the difference between the matching of two different random variables with the same macroscopic state.

	\Cref{def:selfconsistency,def:consistencyprojection,def:weakcontinuitySDE} immediately imply the following corollary:
	\begin{cor}\label{cor:consensemblegeneration}
	Consider a sequence of  matching operators $\left(\overline{\PO}_{[L]}\right)_{L=1,2,\dots}$, and assume that this sequence is consistent for a set of sequences of triples 
\[\left(\Z_{[L]},\X^{t^-,\Z_{[L]}}(t^*),\overline{\Q}_{[L]}\big(\X^{t^-,\Z_{[L]}}(t^*)\big)\right)_{L=1,2,\dots},\]
	where $Z_{[L]}$ are random variables, $t^-=t^*-\Dt$, $\Dt\in[0,t^*-t^0]$, $t^*\in\I$. Suppose further that SDE \eqref{eq:SDE} is uniformly weakly continuous for the set of all $\Z_{[L]}$.
	Then for all $g \in C_P^{0,2}(I\times\R^d, \R)$ there exist constants $\Dt_0>0$ and $C_L$, with $C_L\to0$ for $L\to\infty$, and $L_0$, such that it holds
	\begin{equation}
	 |\E g\big(s,\X^{t^-,\Z_{[L]}}(t^*)\big)-\E g\big(s,\overline{\PO}_{[L]}(\Z_{[L]},\U_{[L]}^*)\big)|\leq C_L\Dt
	\end{equation}
	for all $L\geq L_0$, all $\left(\Z_{[L]},\X^{t^-,\Z_{[L]}}(t^*),\U_{[L]}^*=\overline{\Q}_{[L]}\big(\X^{t^-,\Z_{[L]}}(t^*)\big)\right)_{L=1,2,\dots}$ in the considered set, $\Dt\in[0,\Dt_0]$, and all $t^*\in[t^0+\Dt,T]$, $s\in\I$.
	\end{cor}

	This corollary states that the difference, measured in a weak sense, between the exact distribution at some time instance $t^*$ and a matching from a previous time $t^-=t^*-\Dt$ with the exact macroscopic state at $t^*$ vanishes when matching with more macroscopic state variables or letting $\Dt\to0$.

	In the remainder of the text, we will only use the properties stated above. Therefore, one may use any matching operator that would come to mind for a particular problem, as long as the above properties are satisfied.

\subsection{Matching failure and adaptive time-stepping\label{sec:proj-failure}}

In the numerical experiments, one might encounter situations in which the distributions evolve on time-scales that are not significantly slower than those of the macroscopic functions of interest.  In that case, when taking a large extrapolation time step, the extrapolated macroscopic state differs significantly from the state corresponding to the last available polymer ensemble, and it is possible that matching the ensemble with the desired macroscopic state fails. Consequently, a failure in the matching can be used as an indication to decrease the extrapolation time step.  Based on this observation, we propose the following criterion to adaptively determine the macroscopic step size $\Dt$: If the matching fails, we reject the step and try again with a time step
\begin{equation}\label{eq:ad_1}
 \Dt_{\text{new}}=\max(\underline{\alpha} \Dt, K\dt), \qquad\underline{\alpha} < 1,
\end{equation}
 whereas, when the matching succeeds, we accept the step and propose
\begin{equation}\label{eq:ad_2}
 \Dt_{\text{new}}=\min(\overline{\alpha} \Dt,\Dt_{\text{max}}),\qquad \overline{\alpha}>1
 \end{equation}
 for the next step. If the macroscopic step size $\Dt_{\text{new}} = K\dt$, matching becomes trivial (the identity operator), since there is no extrapolation. Note that, when this happens, the criterion will ensure that the larger time steps are tried after the next burst of microscopic simulation.


\section{Matching operator\label{sec:proj}}

In this section, we analyze the matching operator in more detail.  We first define the specific matching operator that will be considered in this paper (\cref{sec:PropertiesProjectionstep}). We then prove a consistency result in a very special case (\cref{sec:prop-proj-result}) and give and illustrate then a conjecture for the general case (\cref{sec:prop-proj-conj}).

\subsection{Definition of specific matching operator\label{sec:PropertiesProjectionstep}}

The matching operator $\PO$ can be defined in several ways. One option is to minimize the difference between the given ensemble and the result of the matching, for instance in the $2-norm$, i.\,e.,
\begin{equation}\label{eq:AnsatzVertBerech}
\bY^{{\tind}+1}=\argmin_{\bZ:~\Q(\bZ)=\U^{{\tind}+1}}\frac{1}{2}\|\bZ -\bY^{{\tind},K}\|^2.
\end{equation}
An approximation of the resulting ensemble can be obtained using a Lagrange multiplier technique:
\begin{equation}\label{eq:proj1}
\left\{ \begin{aligned} &\bY^{{\tind}+1}=\bY^{{\tind},K}+\sum_{l=1}^L
\lambda_l \nabla_{\bY} \Q_l (\bY^{{\tind},K}),\\ &\text{ with
$\Lambda=\{\lambda_l\}_{l=1}^L$ such that $\Q_l (\bY^{{\tind}+1}) =
U_l^{{\tind}+1}$ for $l=1,\ldots,L$.} \end{aligned} \right. \end{equation}
Then, $\bY^{{\tind}+1}=\PO(\bY^{{\tind},K},\U^{{\tind}+1})$ results after
obtaining $\Lambda$ from a Newton procedure that solves the $L$-dimensional nonlinear system that defines the constraints.
(Note that one can show that the resulting ensemble exactly satisfies \eqref{eq:AnsatzVertBerech} if we write an implicitly defined
gradient $\nabla_{\bY} \Q_l (\bY^{{\tind}+1})$ in the first line of
\eqref{eq:proj1}.)\\
By applying the implicit function theorem, we obtain the following lemma:
\begin{lem}
The problem \eqref{eq:proj1} has a locally unique solution if
\begin{itemize}
\item in the case of standard empirical moments $U_l^n$,
\begin{equation}\label{eq:conduniquesolstempmom}
\det\left(U_{i+k-2}^n\right)_{i,k=1,\dots,L}\neq0,
\end{equation}
\item in the case of empirical centralized moments $U_l^n$,
\[
\det\left(U_{i+k-2}^n-U_{i-1}^nU_{k-1}^n\right)_{i,k=2,\dots,L}\neq0.
\]
\end{itemize}
\end{lem}%
\begin{remark}The determinant
\[
\det\left(\E X^{i+k-2}(t^n)\right)_{i,k=2,\dots,L},
\]
obtained by replacing the empirical moments in \eqref{eq:conduniquesolstempmom} by the moments themselves, is, from the theory of the Hamburger moment problem, known to be positive for distributions with finite support.
\end{remark}

Alternative matching operators can be defined by measuring the difference between random variables differently, for instance using the Kullback-Leibler divergence (relative entropy) \cite{KullbackLeibler51}. For some choices of the macroscopic state variables $\U$ one
could also consider classical moment matching, as described in
\cite{caflisch98mca}.

\begin{remark}[Matching for the FENE dumbbells] For the FENE dumbbells, an accept-reject strategy is applied during the combined evolution and matching, i.\,e., if the state of a polymer would become unphysical during the matching, we reject the trial move for that specific polymer in the evolution step and repeat the time step for this polymer, after which the matching of the ensemble is tried again.
\end{remark}

\subsection{Particular results for normal distributions\label{sec:prop-proj-result}}
For simplicity of the argument, and without loss of generality, we restrict ourselves to a one-dimensional notation for the rest of this section.
In the case of scalar normally distributed random variables, the following result can be proved.
\begin{lem}\label{lem:normproject}
Consider a scalar random variable $Z$ and suppose that the macroscopic state variables are given by $u_1(z)=z$ and $u_2(z)=(z-U_1)^2$. Suppose further that in analogy to \eqref{eq:AnsatzVertBerech}, $\overline{\PO}$ is given by
\begin{equation}\label{eq:AnsatzVertBerechideal}
\overline{\PO}(Z,\U^*)=\argmin_{\overline{\Q}(Z^*)=\U^*}\frac{1}{2}\E\left((Z^*-Z)^2\right).
\end{equation}
Then, the random variable $\overline{\PO}(Z,\U^*)$ is also normally distributed with mean and variance given by $\U^*$.
\end{lem}
\begin{proof}Denote by
\[
	\begin{pmatrix}
	\mu\\\sigma^2
	\end{pmatrix}
	=
	\overline{\Q}(Z),\quad
	\begin{pmatrix}
	\mu^*\\(\sigma^*)^2
	\end{pmatrix}
	=
	\U^*,\quad\sigma,\sigma^*>0.
	\]
Then, \eqref{eq:AnsatzVertBerechideal} yields
\begin{equation}\label{eq:Ynpsend}
\overline{\PO}(Z,\U^*)=\pm\sqrt{\frac{(\sigma^*)^2}{\sigma^2}}\left(Z-\mu\pm \mu^*\sqrt{\frac{\sigma^2}{(\sigma^*)^2}}\right).
\end{equation}
Thus, if $Z$ is normally distributed, so is $\overline{\PO}(Z,\U^*)$.
\end{proof}

With the help of this lemma, we can easily show the following two corollaries.
\begin{cor}
Suppose that the hierarchy of macroscopic state variables is defined using
$u_1(z)=z$ and $u_l(z)=(z-U_1)^l$ for $l\geq2$. Suppose further that $\overline{\PO}$ is given by
\eqref{eq:AnsatzVertBerechideal}.
Then, the sequence of matching operators is consistent with $C_L=0$ for $L\geq2=L_0$ for all sequences of triples $\left(Z_{[L]}^+,Z_{[L]}^-,\U_{[L]}^*\right)_{L=1,2,\dots}$, where $Z_{[L]}^+$ and $Z_{[L]}^-$ are normally distributed and $\left(\U_{[L]}^*\right)_{L=1,2,\dots}$ are sequences of centralized moment values consistent with normal distributions.
\end{cor}
\begin{proof}
We first consider the case $L=2$.
As the normal distribution is uniquely determined by its first two (centralized) moments, \cref{lem:normproject} implies that for two normally distributed random variables $Z_1$ and $Z_2$, $\overline{\PO}_{[2]}(Z_1,\U_{[2]}^*)$ and $\overline{\PO}_{[2]}(Z_2,\U_{[2]}^*)$ are identically distributed, and thus $C_2=0$. For the same reason, \eqref{eq:Ynpsend} holds also for $L>2$, and also in this case $C_L=0$, and the matching is consistent.
\end{proof}
\begin{cor}
Suppose that the hierarchy of macroscopic state variables is defined using
$u_1(z)=z$ and $u_l(z)=(z-U_1)^l$ for $l\geq2$. Suppose further that $\overline{\PO}$ is given by
\eqref{eq:AnsatzVertBerechideal}.
Then, the sequence of matching operators is continuous for all normally distributed random variables, and sequences of centralized moment values consistent with normal distributions.
\end{cor}
\begin{proof}
For $L\geq2$, \cref{lem:normproject} implies again that $Z^*=\overline{\PO}_{[L]}(Z,\U^*_{[L]})$ and $Z^+=\overline{\PO}_{[L]}(Z,\U^+_{[L]})$ are normally distributed if $Z$ is normally distributed and $\U^*_{[L]}$, $\U^+_{[L]}$ are sequences of centralized moment values consistent with normal distributions. Denoting the corresponding expectations by $\mu^*$, resp.~$\mu^+$, and variances by
$(\sigma^*)^2$, resp.~$(\sigma^+)^2$,
it holds for all $f \in C_P^{1}(\R, \R)$
\begin{align*}
|\E f(s,Z^*)-\E f(s,Z^+)|
=&\left|\int_{\R}\big(f(s,\sigma^* z+\mu^*)-f(s,\sigma^+ z+\mu^+)\big)\frac1{\sqrt{2\pi}}e^{-z^2/2}~dz\right|\\
=&\left|\int_{\R}f'(s,\xi_z)[(\sigma^*-\sigma^+)z+\mu^*-\mu^+]\frac{1}{\sqrt{2\pi}}e^{-z^2/2}~dz\right|,
\end{align*}
where $\xi_z\in[\sigma^* z+\mu^*,\sigma^+ z+\mu^+]$. As there exist constants $\tilde{C}$ and $r \in \N$ such that $|f'(s,\xi_z)|\leq\tilde{C}(1+|\max\{\sigma^*,\sigma^+\}z+\max\{\mu^*,\mu^+\}|^{2r})$, this implies also that the sequence of matching operators is continuous (the corresponding \cref{eq:continuity} can be verified similarly in the case $L=1$).
\end{proof}

Several observations can be made. First, one can obtain a similar result whenever the distributions are defined by a finite number of moments (for instance, a lognormal distribution) by taking this knowledge into account when defining the matching operator.
Second, we remark that, if the random variables are normally distributed at all moments in time, this implies that they represent solutions of a linear SDE
in the narrow sense,
\begin{equation}\label{eq:linSDE}
dX(t)=\big(a_1(t)X(t)+a_2(t)\big)~dt+b(t)\star dW(t),\quad t\in\I,
\end{equation}
with normally distributed initial values.  For this equation, it is clear that the complete time evolution of the distributions can be completely described by a system of two ODEs for $U_1=\mu=\E X$ and $U_2=\sigma^2=\E\left(\left(X-\mu\right)^2\right)$, namely
\begin{equation}\label{eq:linear-closed}
	\begin{cases}
	dU_1/dt &= a_1(t)\; U_1 + a_2(t),\\
	dU_2/dt &= 2a_1(t)\; U_2 + b(t)^2.
\end{cases}
\end{equation}
As a consequence, the matching operator \eqref{eq:AnsatzVertBerechideal} corresponds to the reconstruction of the normal distribution corresponding to the given macroscopic state.

\subsection{Conjecture for general distributions\label{sec:prop-proj-conj}}

We now turn to more general distributions.  We assume that the distribution of the random variable is uniquely determined by its moments; this is the case if, for instance, the moment generating function $\sum_{i=0}^\infty\frac{\E(Z^i)t^i}{i!}$ is bounded in an interval around $0$.

In this setting, we propose the following conjecture:
\begin{conjecture}
Consider a sequence of restriction operators $\left({\overline \Q_{[L]}}\right)_{L=1,2,\dots}$ in which the macroscopic state variables corresponding to $\overline \Q_{[L]}$ are defined as the first $L$ centralized moments of the distribution, i.\,e., $u_1(z)=z$, $u_l(z)=(z-U_1)^l$, for $l=2,\ldots, L$, and define the corresponding sequence of matching operators $\left(\overline{\PO}_{[L]}\right)_{L=1,2,\dots}$ via \eqref{eq:AnsatzVertBerechideal}.
Consider further a set $S$ of random variables for which all (centralized) moments of its members exist and uniquely determine the corresponding distribution function, and each moment can be uniformly bounded. Then, the sequence of matching operators is continuous for $S$ and all sequences of macroscopic states $\U_{[L]}$, and consistent for all sequences of triples $\left(Z_{[L]}^*,Z_{[L]}^+,\U_{[L]}\right)_{L=1,2,\dots}$ where $Z_{[L]}^*,Z_{[L]}^+\in S$.
\end{conjecture}

We illustrate the main properties of the matching operator for general distributions by means of numerical experiments. Below, we consider \cref{eq:fene-3d} in one space dimension, with $\kappa(t)\equiv 2$, $F(X)$ the FENE force \eqref{eq:springs} with $\fp=49$, $\We=1$. We discretize in time with the classical Euler-Maruyama scheme with time step $\dt=2\cdot 10^{-4}$, and simulate $\Eind=1\cdot 10^5$ realizations, whose initial state at time $t^0=0$ is taken from the invariant distribution of \cref{eq:fene-3d} for $\kappa(t)\equiv 0$.
As the macroscopic state $\U_{[L]}$, we consider the first $L$ \emph{even} centralized moments, 
since for the exact solution,  the odd (centralized) moments vanish due to symmetry.

\subsubsection{Error dependence on the number of moments}

We first simulate  up to time $t^*=1.15$ and record the microscopic state $\bY^*$ and corresponding macroscopic states $\U^*_{[L]}$ for $L=1,\ldots,10$, at time $t^*$, as well as the microscopic state $\bY^{-}$ at time $t^-=1$.  We then project the ensemble $\bY^{-}$ onto the macroscopic state $\U^*_{[L]}$ and compare the density of $\PO_{[L]}(\bY^{-},\U^*_{[L]})$ with that of $\bY^*$. Since the absolute value of the moments increases quickly with the order of the moment, the residuals in the Newton procedure for the matching are scaled relative to the requested value of the corresponding moment; the Newton iterations are stopped if the norm of the residual is smaller than $1\cdot 10^{-9} $.

We perform three tests to examine the convergence (in empirical density) of $\PO_{[L]}(\bY^{-},\U^*_{[L]})$ to $\bY^*$ for $L\to\infty$.  First, we visually inspect the corresponding empirical probability density functions, see \cref{fig:conjecture-distr-1left}.
\newlength{\figwidth}
\setlength{\figwidth}{0.485 \textwidth}
\begin{figure}[tbp]\hspace*{\fill}
\subfigure[Empirical probability density functions\label{fig:conjecture-distr-1left}]{\includegraphics[scale=0.9]{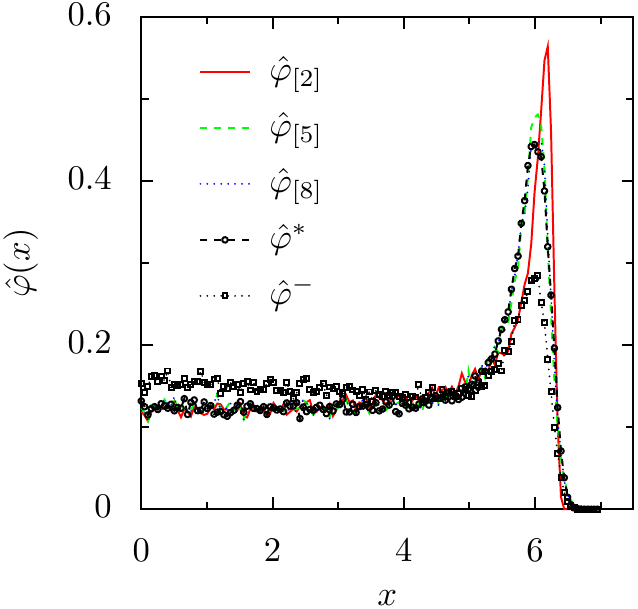}}
\subfigure[Error of the $l$-th moment as a function of $l$\label{fig:conjecture-distr-1right}]{\includegraphics[scale=0.9]{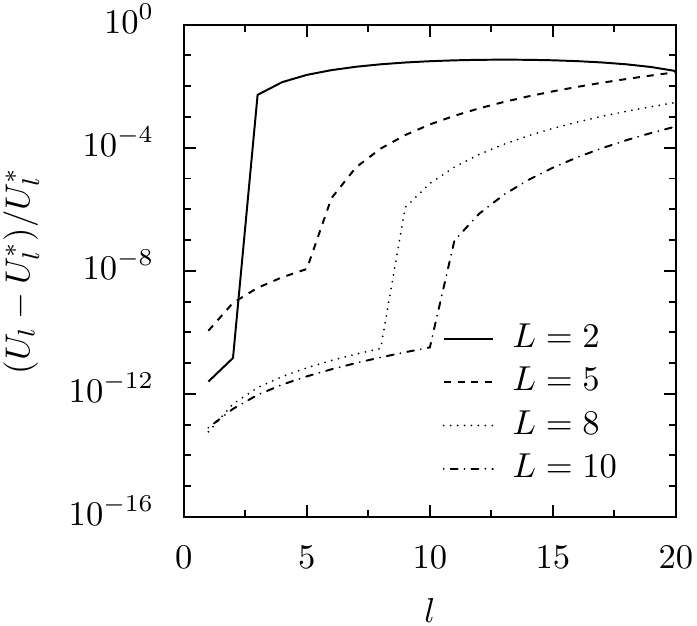}}
\hspace*{\fill}
\caption{\label{fig:conjecture-distr-1}Results after projecting a prior ensemble of FENE dumbbells onto the first $L$ even centralized moments of a reference ensemble for several values of $L$. Simulation details are given in the text.}
\hspace*{\fill}
\end{figure}
Shown are histogram approximations of the empirical density $\hat{\varphi}^{-}$ of $|\bY^-|$ (the initial condition for the matching), the reference empirical density $\hat{\varphi}^*$ of $|\bY^*|$, and approximations $\hat{\varphi}_{[L]}$ of $\PO_{[L]}(\bY^-,\U_{[L]})$ for several values of $L$. The figure visually suggests that, when increasing the number of macroscopic state variables, the reference empirical density gets approximated more accurately. We now take a closer look to the projected ensembles by computing the relative difference between the $l$-th even empirical moment of the projected ensemble, $U_l$, and the corresponding empirical moment of the reference ensemble $U^*_l$ as $(U_l-U^*_l)/U^*_l$. \Cref{fig:conjecture-distr-1right} shows this error as a function of $l$ for different values of the number of macroscopic state variables $L$.
\begin{table}
\begin{tabular}{|c||c|c|c|c|c|c|c|c|c|}
		\hline
		$L$ &2 &3 &4 & 5 & 6 & 7 & 8 & 9 & 10 \\\hline
		$p$ & $0.000$ & $0.000$ & $1.197\cdot 10^{-3}$ & $0.840$ & $0.862$ & $0.999$  & $0.999$ & $0.999$  & $0.999$ \\
		\hline
\end{tabular}
 \caption{\label{fig:conjecture-distr-2}The p-values of a two-sample Kolmogorov--Smirnov test that compares the reference and projected empirical distributions. Simulation details are in the text.}
\end{table}

We make two key observations. First, for $l<L$, the relative difference in the corresponding moment is of the order of the tolerance of the Newton procedure.  This is expected, since these are the macroscopic state variables onto which the distribution is projected.  We see that this very small error increases nevertheless with $l$; this can be explained by pointing out that the value of the moments increases very quickly with $l$, and that the equations in the Newton procedure have been rescaled accordingly.
Second, the error in the higher moments ($l>L$) also decreases with increasing $L$.  This indicates that the convergence of $\hat{\varphi}_{[L]}$ to $\hat{\varphi}^{*}$ for $L\to \infty$ is not only due to the fact that we project onto more moments, but also because the approximation of the higher order moments improves. So far, we have no complete theoretical justification for this observation.

Finally, we compare the reference and projected empirical distributions using a two-sample Kolmogorov--Smirnov test \cite{Lehn85eid}; this classical hypothesis test results in a high p-value ($\leq1$) if the two samples are likely to have been drawn from the same probability distribution. The results are shown in \cref{fig:conjecture-distr-2}.  We clearly see a $p$-value that approaches $1$ for increasing $L$.

\subsubsection{Error dependence on the time step}

In a next experiment, we simulate up to time $t^*=1.54$, and record the microscopic state $\bY^{-}$ at time $t^-=1.5$, as well as the  macroscopic states $\U_{[L]}(t^-+\Dt)$ for $L=3,4,5$, and the function of interest $\tilde{\tau}_p(t^-+\Dt)$ for $\Dt \in [0,t^*-t^-]$. We then project the ensemble $\bY^{-}$ onto the macroscopic state $\U_{[L]}(t^-+\Dt)$, obtaining $\PO_{[L]}\big(\bY^{-},\U_{[L]}(t^-+\Dt)\big)$, and  denote the corresponding value of the stress as $\hat{\tau}_p(t^-+\Dt)$. We record the relative error $|\hat{\tau}_p(t)-\tilde{\tau}_p(t)|/\tilde{\tau}_p(t)$ as a function of $\Dt$.
To reduce the statistical error, we report the averaged results of $100$ realizations of this experiment. The results are shown in \cref{fig:projection-Dt}, where we used $\epsilon=1$ in \eqref{eq:kramers}.
\begin{figure}
\begin{center}
	\includegraphics[scale=0.8]{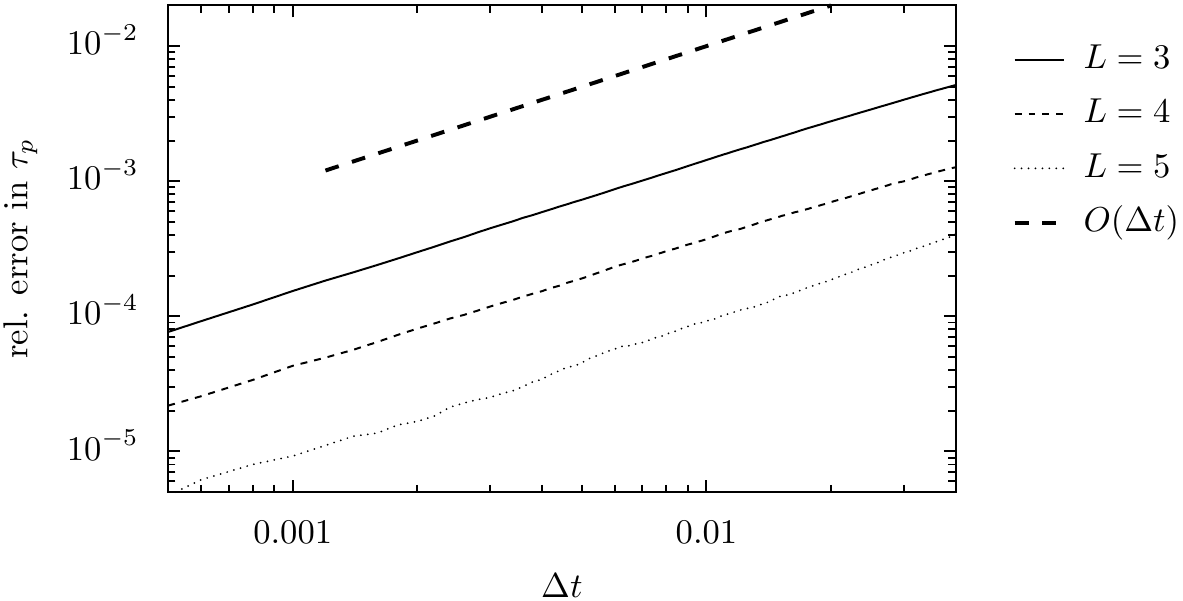}
\end{center}
\caption{\label{fig:projection-Dt}Error of the stress after projecting a prior ensemble of FENE dumbbells onto the first $L$ even centralized moments of a reference ensemble, as a function of $\Dt$ for several values of $L$. Simulation details are given in the text.}
\end{figure}
We indeed see the linear increase of the matching error as a function of $\Dt$; notice also, as was shown in the previous experiment, that the matching error decreases with increasing $L$.  From the figure, we conclude that the remaining statistical error is at least lower than $10^{-5}$.

\section{Extrapolation operator\label{subsec:accel-extrap}}

Next, we need to specify how the extrapolation is performed. Let the order of consistency be defined as follows:

\begin{defi}[Consistency of extrapolation]
Consider a certain class of sufficiently smooth functions.
An extrapolation operator ${\mathcal E}$ is called consistent of order $\pe>0$ for this class if there exist $\Dt_0>0$ and $C$ such that for all $\Dt\in[0,\Dt_0]$, all $\tind\leq\Tind$ and all functions $U$ in the considered class it holds
\begin{equation}\label{eq:ConsistencyExtrap}
\left\|	\tilde{U}^{\tind+1}-U(t^{\tind+1})\right\|\le C \Dt^{p_e+1},
\end{equation}
with
\[
\tilde{U}^{\tind+1} = {\mathcal E}\left(\left(U(t^{i,k}) \right)_{i,k=0,0}^{\tind,K};(\Dt_i)_{i=0}^\tind, \dt \right).
\]
\end{defi}
Further, we will also use the following definition of continuity:
\begin{defi}[Continuity of extrapolation]\label{def:continuityextrapolation}
Consider a certain class of sufficiently smooth functions.
An extrapolation operator ${\mathcal E}$ consistent of order $\pe>0$ for this class is called continuous for this class, if there exist $\Dt_0>0$ and $C$ such that for all $\Dt\in[0,\Dt_0]$, all $\tind\leq\Tind$, and all functions $U_1,U_2$ in the considered class it holds
\begin{equation}\label{eq:ContinuityExtrap}
\begin{split}
\|{\mathcal E}\left(\left(U_1(t^{i,k}) \right)_{i,k=0,0}^{\tind,K};(\Dt_i)_{i=0}^\tind, \dt \right)
-{\mathcal E}\left(\left(U_2(t^{i,k}) \right)_{i,k=0,0}^{\tind,K};(\Dt_i)_{i=0}^\tind, \dt \right)\|\\
\leq C\left(\frac{\Dt}{\dt}\right)^{\pe}\sum_{i,k=0,0}^{\tind,K}\|U_1(t^{i,k})-U_2(t^{i,k})\|.
\end{split}
\end{equation}
\end{defi}
In the remainder of this section, we consider two extrapolation strategies: projective extrapolation (\Cref{sec:proj-extrap}) and multistep state extrapolation (\Cref{sec:msem-extrap}). Numerical illustrations are given in \Cref{sec:num-extrap}.

\subsection{Projective extrapolation\label{sec:proj-extrap}}

The first approach that we consider, proposed in \cite{Gear:2003p2171}, is to extrapolate $\U$ to time $t^{n+1}$ using only (some of) the time points $t^{n,k}$, $k=0,\dots,K$, i.\,e., by using the sequence of points obtained in the last burst of microscopic simulation. If this \emph{coarse projective extrapolation} is based on the interpolating polynomial of degree $\pe$ (with $\pe\leq K$) through the parameter values at times $t^{{\tind},k}$, $k=K-\pe,\dots,K$, we obtain
 \begin{align}\label{eq:PIextrap}
 \U^{{\tind}+1}=\sum_{s=0}^{\pe} l_s(\alpha_{\tind})\U^{{\tind},K-s},
 \end{align}
 with \begin{equation}\label{eq:alpha}
 \alpha_{\tind}=\frac{\Dt_{\tind}}{\dt}-K,
 \end{equation} and the Lagrange polynomials
 \begin{equation}\label{eq:lsm}
l_s(\alpha)=\frac{\alpha(\alpha+1)\cdots(\alpha+{\pe})}{s!({\pe}-s)!(-1)^s(\alpha+s)}.
 \end{equation}
Clearly, \eqref{eq:PIextrap} is continuous in the sense of \cref{def:continuityextrapolation}.

\begin{example}[Coarse projective forward Euler]\label{ex:coarseprojectiveforwardEuler}
The simplest, first order version of the above method is called coarse projective forward Euler. In this case, the procedure can be rewritten as
\begin{equation}\label{eq:cpfe}
	\U^{{\tind}+1} = \U^{{\tind},K}+(\Dt_{\tind} - K\dt)\bs{\overline{\Hc}}^{\tind}, \qquad\bs{\overline{\Hc}}^{\tind} = \frac{\U^{{\tind},K}-\U^{{\tind},K-1}}{\dt}.
\end{equation}

\end{example}

The procedure described above is reminiscent of a Taylor method \cite{Gear:2003p2171}.
Consequently, time integration based on this extrapolation will resemble a Taylor method when repeatedly extrapolating forward in time, and the global deterministic error will be dominated by a term of the form $C\Dt^{p_e}$ (assuming $\dt \ll \Dt$), as results from an accuracy analysis of coarse projective integration for deterministic microscopic models \cite{Vandekerckhove:2008p891}.

To assess qualitatively the influence of coarse projective extrapolation on the statistical error, we apply it to the linear test equation
\begin{equation}\label{eq:staterrorlinSDE}
dX(t)=a X(t)~dt+b dW(t).
\end{equation}
Application of the one step method $\os$ to \eqref{eq:staterrorlinSDE} yields 
\begin{align*}
Y^{{\tind},k}=&\tRos(a,\dt,\eta^{{\tind},k-1})Y^{{\tind},k-1}+\tSos(a,b,\dt,\eta^{{\tind},k-1}),
\end{align*}
where $\tRos$ and $\tSos$ are functions depending on $\os$ and $\eta^{{\tind}+1,k-1}$ are (vectors of) the i.\,i.\,d.\ random variables used by $\os$.
In the following, we assume that $\tRos(a,\dt,\eta^{{\tind},i})$ is independent of $\eta^{{\tind},i}$ and can be written as $\tRos(a,\dt,\eta^{{\tind},i})=\Ros(a\dt)$; this holds, e.\,g., for typical Runge-Kutta methods.
The above assumptions imply
\begin{align*}
Y^{{\tind},k}=&\Ros(a\dt)^kY^{{\tind},0}+\sum_{i=0}^{k-1}\tSos(a,b,\dt,\eta^{{\tind},i})\Ros(a\dt)^{k-i-1},\\
\E Y^{{\tind},k}=&\Ros(a\dt)^k\E Y^{{\tind},0}+\E\tSos(a,b,\dt,\eta^{{\tind},0})\sum_{i=0}^{k-1}\Ros(a\dt)^{i}.
\end{align*}
If we now apply the extrapolation step \eqref{eq:PIextrap}, we obtain
\begin{align}\nonumber
&\E Y^{{\tind}+1,0}=\sum_{s=0}^{\pe}l_s(\alpha)\E Y^{{\tind},K-s}\\\label{eq:staterrorlinsdeonesteppiassump}
=&\underbrace{\left(\sum_{s=0}^{\pe}l_s(\alpha)\Ros(a\dt)^{K-s}\right)}_{=:R
_{\E}(a\dt)}\E Y^{{\tind},0}+\E\tSos(a,b,\dt,\eta^{{\tind},0})\sum_{s=0}^{\pe}l_s(\alpha)\sum_{i=0}^{K-s-1}\Ros(a\dt)^{i},
\end{align}
with $\alpha$ and $l_s(\alpha)$ given by \eqref{eq:alpha} and \eqref{eq:lsm}.
In analogy to \eqref{eq:staterrorlinsdeonesteppiassump} we obtain also
\begin{align*}
\hE Y^{{\tind}+1,0}=R_{\E}(a\dt)\hE Y^{{\tind},0}+\sum_{s=0}^{\pe}l_s(\alpha)\sum_{i=0}^{K-s-1}\Ros(a\dt)^{i}\hE\tSos(a,b,\dt,\eta^{{\tind},i}).
\end{align*}
Thus $\E\hE Y^{{\tind}+1,0}=\E Y^{{\tind}+1,0}$, but
\begin{align}\nonumber
\Var\hE Y^{{\tind}+1,0}=&\frac1{\Eind}R_{\E}(a\dt)^2\Var Y^{{\tind},0}\\\label{eq:VarhECPI}
&+\frac1{\Eind}\Var\tSos(a,b,\dt,\eta^{{\tind},0})\sum_{i=0}^{K-1}\Ros(a\dt)^{2i}\left(\sum_{s=0}^{\min\{\pe,K-1-i\}}l_s(\alpha)\right)^2.
\end{align}
The second summand in \eqref{eq:VarhECPI} behaves as $\dt \alpha^{2\pe}/\Eind$ for large $\alpha=\Dt/\dt - K$.  Assuming $\Dt \gg \dt$, this results in an amplification of the statistical error with a factor
$\Dt^{p_e}/\dt^{p_e}$ during extrapolation. A natural question is then: how many realizations $\tilde{\Eind}$ are needed to obtain the same variance using a fully microscopic simulation? In that case, we have
\begin{align}\nonumber
\Var\hE Y^{{\tind},\alpha+K}=&\frac1{\tilde\Eind}\Ros(a\dt)^{2(\alpha+K)}\Var Y^{{\tind},0}\\\label{eq:VarhECPIfull}
&+\frac1{\tilde\Eind}\Var\tSos(a,b,\dt,\eta^{{\tind},0})\sum_{i=0}^{\alpha+K-1}\Ros(a\dt)^{2i}.
\end{align}
Thus, for large $\alpha$, the required number of realizations for a full microscopic simulation is smaller by a factor $1/\alpha^{2\pe-1}$, i.e, $\tilde{\Eind}\sim\frac{\Eind}{\alpha^{2\pe-1}}$, whereas the computational costs per realization increases by a factor $\alpha$. This means that for large $\alpha$ and $\pe=1$, the computational cost of the micro/macro acceleration technique with coarse projective extrapolation is similarly to that of a full microscopic simulation for a given variance.  For large $\alpha$ and $\pe>1$, coarse projective extrapolation is even more expensive than a full microscopic simulation.


To reduce statistical error, it has been proposed to use a chord based approximation, for instance, using $\U^{{\tind},K-K_1}$ for the time derivative estimate instead of $\U^{{\tind},K-1}$ in equation~\eqref{eq:cpfe} \cite{RicoGearKevr04}. Instead of taking Lagrange polynomials in equation~\eqref{eq:PIextrap}, we then have
\begin{equation}\U^{n+1}=\sum_{s=0}^{K}l_s(\alpha_n)\U^{n,K-s},
\end{equation}
in which $l_{K}(\alpha)=1+\frac{\alpha}{K-K_1}$, $l_{K_1}(\alpha)=-\frac{\alpha}{K-K_1}$, and $l_s(\alpha)=0$ otherwise (see \cref{ex:coarseprojectiveforwardEuler}) reduces the variance by a factor $1/(K-K_1)$.  However, the conclusion on the computational cost remains the same.

\subsection{Multistep state extrapolation\label{sec:msem-extrap}}

Because of the amplification of statistical error, we look into alternative extrapolation strategies.  One approach, proposed in \cite{SOMMEIJER:1990p2657,vandekerckhove07nsa}, is to extrapolate $\U$ to time $t^{{\tind}+1}$ using only (some of) the time points $t^{i,K}$, $i=1,\dots,{\tind}$, i.\,e., by using \emph{the last point of each sequence of microscopic simulations}, instead of a sequence of points from the last microscopic simulation. If this \emph{multistep state extrapolation method} is based on the interpolating polynomial of degree ${\pe}$ through the parameter values at times $t^{i,K}$, $i={\tind}-{\pe},\dots,{\tind}$, and we assume equidistant coarse time steps $\Dt$, we obtain
 \begin{align}\label{eq:MSextrap}
 \U^{{\tind}+1}=\sum_{s=0}^{\pe}l_s(\beta)\U^{{\tind}-s,K},
 \end{align}
 where
 \begin{equation}\label{eq:beta}
 \beta=\frac{\alpha}{\alpha+K}
 \end{equation} is the fraction of the interval $\Dt$ over which we extrapolate, and $\alpha$ and $l_s$ are defined as in \eqref{eq:alpha} and \eqref{eq:lsm}. Note that such an extrapolation strategy requires a separate starting procedure.

For a detailed comparison of the accuracy and stability properties of acceleration of the numerical integration of ODEs using projective extrapolation and multistep state extrapolation, we refer to \cite{Vandekerckhove:2008p891,vandekerckhove07nsa}.
Here, we only remark that, while the local error using multistep state extrapolation only differs by a factor two with respect to the local error of projective extrapolation, the global error is affected quite significantly.  This is due to the fact that, when increasing $\Dt$, also the time derivative estimate itself is taken over a larger time interval.  It has been shown in \cite{SOMMEIJER:1990p2657,Vandekerckhove:2008p891} that this results in an error constant (see, e.\,g., \cite[Section III.2]{HairerWanner}) of the form $C\alpha\Dt^{p_e}$. This amplification effect will be illustrated in \cref{sec:num-extrap}.

Let us now look into the statistical error, again using the linear test equation \eqref{eq:staterrorlinSDE}. One extrapolation step of coarse multistep state extrapolation yields
\begin{align*}
\hE Y^{{\tind}+1,0}=&\sum_{s=0}^{\pe}l_s(\beta)\hE Y^{{\tind}-s,K}.
\end{align*}
Thus again $\E\hE Y^{{\tind}+1,0}=\E Y^{{\tind}+1,0}$, but now
\begin{align*}
\Var\hE Y^{{\tind}+1,0}\leq\frac1{\Eind}\max_{s=0}^{\pe}\Var Y^{{\tind}-s,K}\left(\sum_{s=0}^{\pe}|l_s(\beta)|\right)^2.
\end{align*}
As $\beta<1$, the last factor can be bounded (independently of $\alpha$). Consequently, the amplification of statistical error during the extrapolation does not depend on $\alpha$, whereas the corresponding computational costs per simulation path are reduced by a factor $\alpha$ compared to a full microscopic simulation.

\subsection{Numerical illustration\label{sec:num-extrap}}

We now provide a numerical result to illustrate the effects of extrapolation on the deterministic and statistical error.  To avoid effects of the matching step, we consider the linear equation \eqref{eq:linSDE} with $a_2(t)=-a_1(t)=b(t)\equiv 1$, for which we know that macroscopic evolution closes in terms of the first two moments of the distribution.  This microscopic SDE is discretized using an Euler-Maruyama scheme with $\delta t=2\cdot 10^{-4}$. We consider $500$ realizations of a computational experiment with $J=1000$ SDE realizations. As an initial condition, we sample from a standard normal distribution. We compare the sample mean behavior and sample standard deviation of a full microscopic simulation (which we will call the reference simulation) with the micro/macro acceleration algorithm using  $\Delta t=1\cdot 10^{-3}$, $2\cdot 10^{-3}$, $4\cdot 10^{-3}$, and $8\cdot 10^{-3}$. The function of interest is chosen to be $\bar{f}(t)=\E(X(t)^2)$. We denote by $\tilde{f}(t)$ the approximation to the function of interest calculated from one realization of the reference simulation using $J$ SDE realizations, and by $\hat{f}(t)$ the function of interest obtained via one realization of the micro/macro acceleration technique. As extrapolation techniques, we use first order projective extrapolation and first and second order multistep state extrapolation.

\Cref{fig:stat-lin-proj} shows the results for first order projective extrapolation.
\begin{figure}
\begin{center}
	\includegraphics[width=\linewidth]{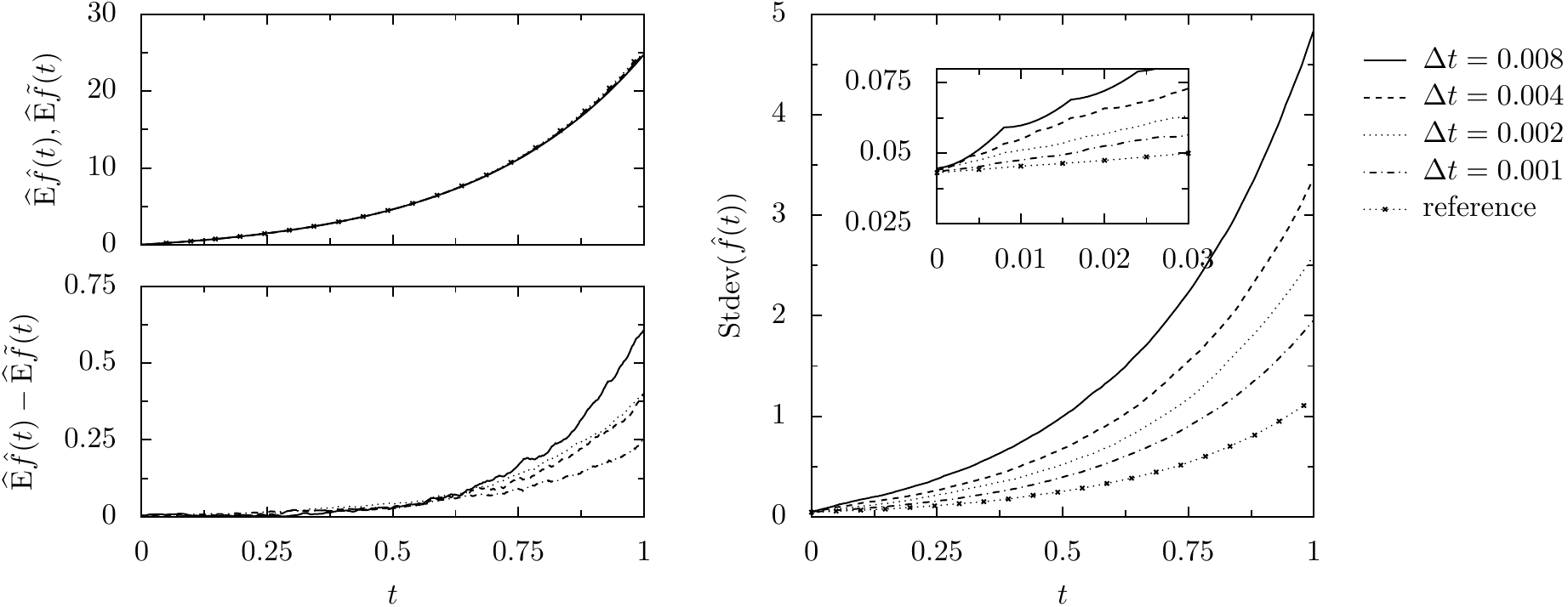}
	\end{center}
\caption{\label{fig:stat-lin-proj}Results of micro/macro acceleration of the linear equation \eqref{eq:linSDE} using $L=2$ moments and projective extrapolation for different values of the time step $\Delta t$, as well as a  full microscopic (reference) simulation. Top left: evolution of the sample means of the function of interest $\tilde{f}$ and $\hat{f}$. Bottom left: deterministic error on $\hat{f}$. Right: evolution of the sample standard deviation of $\hat{f}$. Simulation details are given in the text.}
\end{figure}
The left figure clearly shows, as expected, that the deterministic error grows with increasing $\Delta t$. However, from the right figure follows that, for an individual realization of the experiment, the error is dominated by the statistical error.  The zoom shows that, for small $t$, the sample standard deviation grows linearly as a function of time, with a slope that is larger for larger $\Delta t$.  This is in agreement with the theoretical result on the local propagation of statistical error.

Next, we look at first order multistep state extrapolation, for which the results are shown in \cref{fig:stat-lin-msem-1}.
\begin{figure}
\begin{center}
	\includegraphics[width=\linewidth]{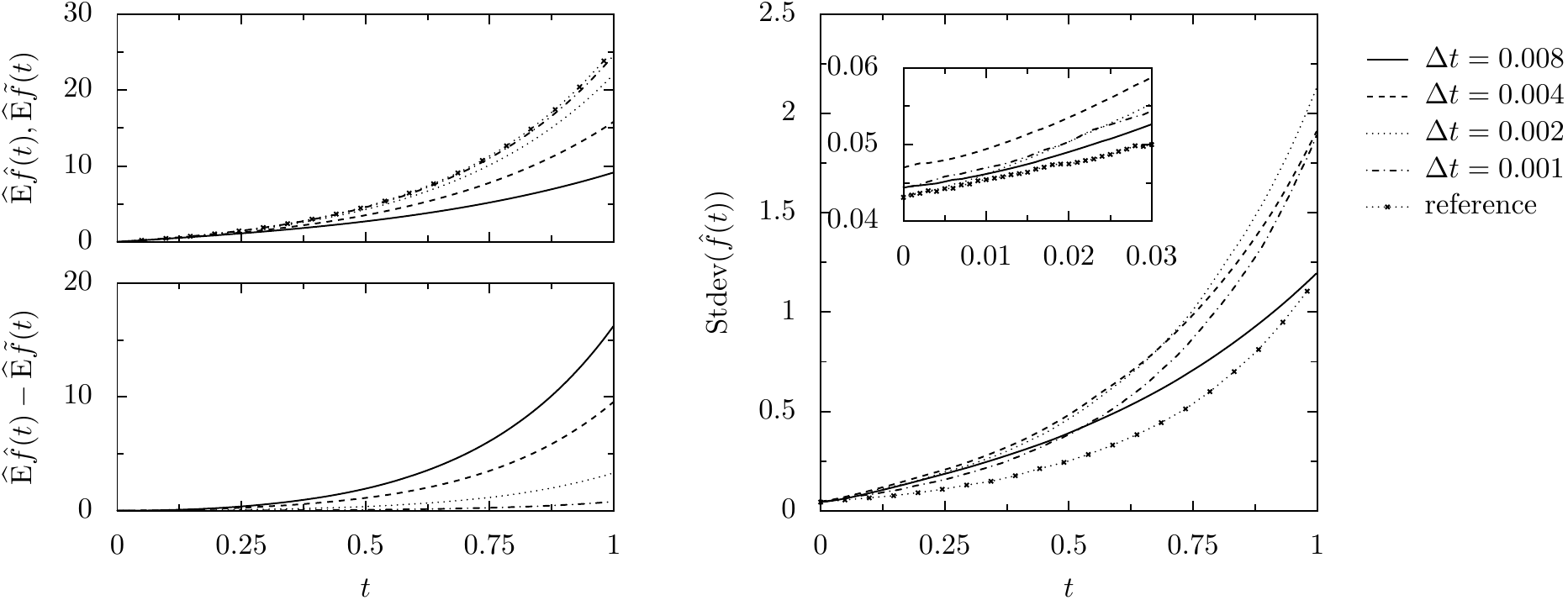}
	\end{center}
\caption{\label{fig:stat-lin-msem-1}Results of micro/macro acceleration of the linear equation \eqref{eq:linSDE} using $L=2$ moments and first order multistep state extrapolation for different values of the time step $\Delta t$, as well as a  full microscopic (reference) simulation.  Top left: evolution of the sample means of the stresses $\tilde{f}$ and $\hat{f}$. Bottom left: deterministic error on $\hat{f}$. Right: evolution of the sample standard deviation of $\hat{f}$. Simulation details are given in the text.}
\end{figure}
The left figure indicates that, when comparing with projective extrapolation, the deterministic error grows much more rapidly with increasing $\Delta t$. On the right, we see that, while the sample standard deviation is larger than for the reference simulation, the sample standard deviation does not depend crucially on $\Delta t$, as is also expected from the analysis of the local propagation of statistical error.  Note that the lower sample standard deviation for $\Delta t=8\cdot 10^{-3}$ is related to the fact that $\hat{f}(t)$ itself is much lower as a consequence of the large deterministic error (see left figure). The zoom shows that, for small $t$, the statistical error grows linearly as a function of time, with a slope that is independent of $\Delta t$, and is identical to the slope for the reference simulation.  These results are in agreement with the theoretical result on the local propagation of statistical error.

Finally, we consider second order multistep state extrapolation.  The results are shown in \cref{fig:stat-lin-msem-2}.
\begin{figure}
\begin{center}
	\includegraphics[width=\linewidth]{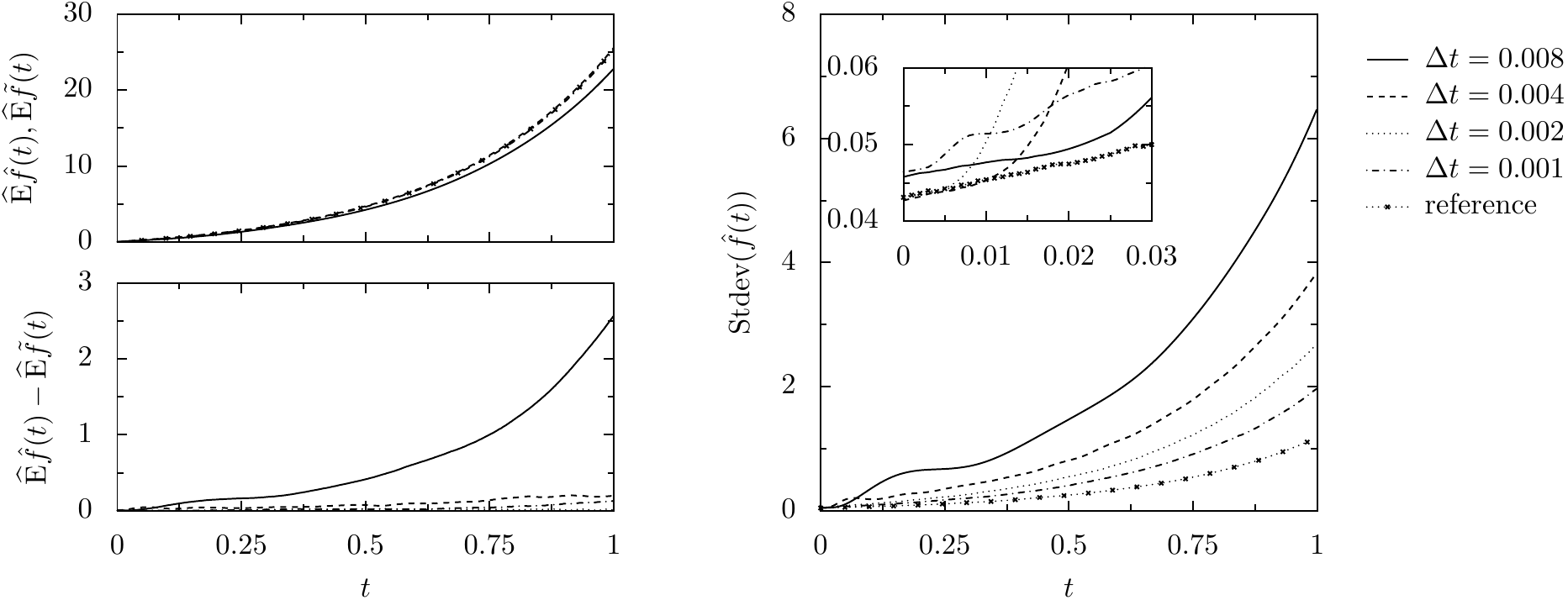}
	\end{center}
\caption{\label{fig:stat-lin-msem-2}Results of micro/macro acceleration of the linear equation \eqref{eq:linSDE} using $L=2$ moments and second order multistep state extrapolation for different values of the time step $\Delta t$, as well as a reference fully microscopic simulation.   Top left: evolution of the sample means of the function of interest $\tilde{f}$ and $\hat{f}$. Bottom left: deterministic error on $\hat{f}$. Right: evolution of the sample standard deviation of $\hat{f}$. Simulation details are given in the text.}
\end{figure}
The left figure shows that the deterministic error is much better than for the first order version.  However, the behavior of the statistical error is more intriguing. When zooming in to the behavior for small $t$, we observe that, over a short time interval, the sample standard deviation  using the micro/macro acceleration technique increases at the same rate as the sample standard deviation in the reference simulation. This corresponds to the theoretical result on the local propagation of statistical error. However, on longer time scales, the sample standard deviation for large $t$ grows rapidly, and seems to be larger for larger $\Delta t$.  As such, second order multistep state extrapolation behaves similarly to first order projective integration on long time scales. We suspect that the loss of this favorable error propagation is due to accumulation effects.  This is indicated by the fact that the length of the time interval on which the local theoretical results is observed is longer for larger values of $\Delta t$.  Indeed, the figure indicates that the effects of accumulated statistical errors start to appear \emph{after a given number of extrapolations}, independently of the size of the extrapolation step.  This behavior requires additional analysis.

\section{Convergence results\label{sec:conv}}

Using the above, we are now ready to give a definition of convergence for the proposed algorithm:

\begin{defi}
Consider a sequence of restriction operators $\left(\overline{\Q}_{[L]}\right)_{L=1,2,\dots}$ and a sequence of matching operators $\left(\overline{\PO}_{[L]}\right)_{L=1,2,\dots}$, and denote the corresponding numerical approximation process obtained by using $\overline{\Q}_{[L]}$ and $\overline{\PO}_{[L]}$ in \Cref{algo:accel} by $\Y_{[L]}$, and the maximum step size by $\Dt$, $\Dt=\max_{{\tind}=1}^{\Tind}\Dt_{\tind}$. The accelerated micro/macro Monte Carlo simulation is then called weakly convergent to the solution $\X$ of SDE \eqref{eq:SDE} as $\Dt \rightarrow 0$ and $L\to\infty$ at any time $t \in {\IDt}$ with time order $p$ if for each $f \in
    C_P^{2(p+1)}(\R^d, \R)$ there exist constants  $\Dt_0>0$, $L_0$, $C_L$, and $\tC_L$, with $C_L\to0$ for $L\to\infty$, such that
\begin{equation}
| \E f\big(\Y_{[L]}(t)\big) - \E f\big(\X(t)\big) |\leq C_L+\tC_L(\Dt)^{p}
\end{equation}
holds for all $t\in\IDt$, all $L\geq L_0$, and all $\Dt\in[0,\Dt_0]$.
\end{defi}

We first discuss convergence when extrapolation is performed as in coarse projective integration (\cref{sec:conv-cpi}). Due to the multistep nature of the extrapolation, proving convergence for the multistep state extrapolation method is more involved; \cref{sec:conv-msem} contains a result for a linear SDE.

\subsection{Convergence using projective extrapolation\label{sec:conv-cpi}}

The following theorem generalizes the theorem for the convergence of one step methods due to Milstein (see \cite{milstein95nio,milstein04snf} or also \cite{debrabant08cwa}).
\begin{theorem}
Suppose the following conditions hold:
\begin{enumerate}[(i)]
\item \label{St-lg-cond1} The coefficient functions $\ab(\x)$ and $\bb^i(\x)$ (where $\bb^i$ denotes the $i$-th column of $\bb$) are continuous, satisfy a Lipschitz condition with respect to $\x$, and belong to $C_P^{p+1,2(p+1)}(\I\times\R^d, \R)$, $i=1,\dots,m$.
For non It\^{o} SDEs, we require in addition that $b^i$ is differentiable and that also ${b^i}'b^i$ satisfies a Lipschitz condition and belongs to $C_P^{p,2(p+1)}(I\times\R^d,\R)$, $i=1,\dots,m$.
\item \label{St-lg-cond2} For sufficiently large $r$ the moments $\E(\|\Y_{[L]}^{\tind,k}\|^{2r})$ exist for $k=0,\dots,K$ and ${\tind}=0,1,\ldots,\Tind$ and are uniformly bounded with respect to $L$, ${\Tind}$.
\item SDE \eqref{eq:SDE} is uniformly weakly continuous for the set of all $\Y_{[L]}^{\tind,K}$.
\item \label{St-lg-cond3} The one step method $\os$ is weakly consistent of order $\pos$.
\item \sloppy The sequence of matching operators is continuous for the numerical approximation process and all sequences of macroscopic states, and consistent for all sequences of triples $\left(\Y_{[L]}^{\tind,K},\X^{t^{\tind,K},\Y_{[L]}^{\tind,K}}(t^{\tind+1,0}),\overline{\Q}_{[L]}(\Y_{[L]}^{\tind+1,0})\right)_{L=1,2,\dots}$.
\item The extrapolation is consistent of order $\pe\geq1$ and continuous for the class of all functions $U(t)=\overline{\Q}_{[L]}\big(\X^{t^{i,k},\Y_{[L]}^{i,k}}(t)\big)$.
\end{enumerate}
Then, the micro/macro acceleration algorithm with projective extrapolation is weakly convergent with time order $p=\min\{\pe,\pos\}$.
\end{theorem}
\begin{proof}
Let $g(s,\x):=\E\Big(f\big(\X(t^{{\tind}+1})\big)|\X(s)=\x\Big)$ for $s\in\I$, $\x\in\R^d$, and $t^{{\tind}+1}\in\IDt$ with $s \leq t^{{\tind}+1}$. Due to condition~(\ref{St-lg-cond1}) $g\in C_P^{0,2(p+1)}$ \cite{milstein95nio}. Therefore, the consistency of $\os$ implies that $g$ satisfies
\begin{equation}\label{eq:Konsuuniform}
    |\E g\big(s ,\X^{t,\x}(t+\dt)\big) -\E g\big(s ,\os(t,\x;\dt)\big)|\leq C_g(\x) \,\dt^{\pos+1}
\end{equation}
uniformly w.\,r.\,t.\ $s \in [t^0, t^{{\tind}+1}]$ for some $C_g\in C_P^{0}(\R^d,\R)$.
Analogously, \cref{cor:consensemblegeneration} and the continuity of the matching imply that $g$ satisfies
\begin{align*}
&|\E g\big(s,\X^{t^{i,K},\Y_{[L]}^{i,K}}(t^{i+1,0})\big)-\E g\big(s,\Y_{[L]}^{i+1,0}\big)|
\\&
\leq|\E g\big(s,\X^{t^{i,K},\Y_{[L]}^{i,K}}(t^{i+1,0})\big)
-\E g\bigg(s,\overline{\PO}_{[L]}\Big(\Y_{[L]}^{i,K},\overline{\Q}_{[L]}\big(\X^{t^{i,K},\Y_{[L]}^{i,K}}(t^{i+1,0})\big)\Big)\bigg)|
\\&
+|\E g\bigg(s,\overline{\PO}_{[L]}\Big(\Y_{[L]}^{i,K},\overline{\Q}_{[L]}\big(\X^{t^{i,K},\Y_{[L]}^{i,K}}(t^{i+1,0})\big)\Big)\bigg)
-\E g\Big(s,\overline{\PO}_{[L]}\big(\Y_{[L]}^{i,K},\overline{\Q}_{[L]}(\Y_{[L]}^{i+1,0})\big)\Big)|
\\&
\leq C_L\Dt+C_1\|\overline{\Q}_{[L]}\big(\X^{t^{i,K},\Y_{[L]}^{i,K}}(t^{i+1,0})\big)-\overline{\Q}_{[L]}(\Y_{[L]}^{i+1,0})\|.
\end{align*}
The last summand can be expanded as follows:
\begin{align*}
&{\|\overline{\Q}_{[L]}\big(\X^{t^{i,K},\Y_{[L]}^{i,K}}(t^{i+1,0})\big)-\overline{\Q}_{[L]}(\Y_{[L]}^{i+1,0})}\|
\\
\leq&
\|\overline{\Q}_{[L]}\big(\X^{t^{i,K},\Y_{[L]}^{i,K}}(t^{i+1,0})\big)-\overline{\Q}_{[L]}\big(\X^{t^{i,0},\Y_{[L]}^{i,0}}(t^{i+1,0})\big)\|
\\&+
\|\overline{\Q}_{[L]}\big(\X^{t^{i,0},\Y_{[L]}^{i,0}}(t^{i+1,0})\big)-{\mathcal E}\left(\left( \overline{\Q}_{[L]}\big(\X^{t^{i,0},\Y_{[L]}^{i,0}}(t^{i,k})\big) \right)_{k=0}^{K};\Dt_i, \dt \right)\|
\\&+
\|{\mathcal E}\left(\left( \overline{\Q}_{[L]}\big(\X^{t^{i,0},\Y_{[L]}^{i,0}}(t^{i,k})\big) \right)_{k=0}^{K};\Dt_i, \dt \right)-{\mathcal E}\left(\left(\overline{\Q}_{[L]}\big( \Y_{[L]}^{i,k}\big)\right)_{k=0}^{K};\Dt_i, \dt \right)
\|.
\end{align*}
Let $\tu_{[L]}(s,\x;t):=\bigg(\E\Big(g_l\big(\X(t)\big)|\X(s)=\x\Big)\bigg)_{l=1}^L$ for $s\in\I$, $\x\in\R^d$, and $t\in\IDt$ with $s \leq t$. With this definition, the continuity and consistency of the extrapolation imply
\begin{align*}
&{\|\overline{\Q}_{[L]}\big(\X^{t^{i,K},\Y_{[L]}^{i,K}}(t^{i+1,0})\big)-\overline{\Q}_{[L]}(\Y_{[L]}^{i+1,0})}\|
\\
\leq&
\|\overline{\Q}_{[L]}\big(\X^{t^{i,K},\Y_{[L]}^{i,K}}(t^{i+1,0})\big)-\overline{\Q}_{[L]}\big(\X^{t^{i,0},\Y_{[L]}^{i,0}}(t^{i+1,0})\big)\|
+C_2(\Dt)^{\pe+1}
\\&+C_3\left(\frac{\Dt}\dt\right)^{\pe}\sum_{k=1}^K\|\overline{\Q}_{[L]}\big(\X^{t^{i,0},\Y_{[L]}^{i,0}}(t^{i,k})\big)-\overline{\Q}_{[L]}\big( \Y_{[L]}^{i,k}\big)\|
\\
=&
\|\E\tu_{[L]}\big(t^{i,K},\Y_{[L]}^{i,K};t^{i+1,0}\big)-\E\tu_{[L]}\big(t^{i,K},\X^{t^{i,0},\Y_{[L]}^{i,0}}(t^{i,K});t^{i+1,0}\big)\|
+C_2(\Dt)^{\pe+1}
\\&+C_3\left(\frac{\Dt}\dt\right)^{\pe}\sum_{k=1}^K\|\left(\E g_l\big(\X^{t^{i,0},\Y_{[L]}^{i,0}}(t^{i,k})\big)-\E g_l\big( \Y_{[L]}^{i,k}\big)\right)_{l=1}^L\|,
\end{align*}
where we made also use of $\X^{t^{i,0},\Y_{[L]}^{i,0}}(t^{i+1,0})=\X^{t^{i,K},\X^{t^{i,0},\Y_{[L]}^{i,0}}(t^{i,K})}(t^{i+1,0})$.
Due to the consistency of $\os$, altogether we obtain
\begin{align}
|\E g\big(s,\X^{t^{i,K},\Y_{[L]}^{i,K}}(t^{i+1,0})\big)-\E g\big(s,\Y_{[L]}^{i+1,0}\big)|
\leq C_L\Dt+\tC_L(\Dt)^{\pe}\dt^{\pos}+\tC_L\dt^{\pos+1}\label{eq:KonsuEnsembleGen}
\end{align}
uniformly w.\,r.\,t.\ $s \in [t^0, t^{{\tind}+1}]$ for some constants $C_L$ and $\tC_L$ with $C_L\to0$ for $L\to\infty$.
For ease of notation, in the following we will neglect the $L$-dependency of $Y$.
Then
\begin{align*}
&{\E f\big(\X^{t^0,X_0}(t^{{\tind}+1})\big)-\E f\big(\Y^{{\tind}+1,0}\big)}\\
=&\sum_{i=0}^{{\tind}}\sum_{k=0}^{K-1}
\left(\E f\big(\X^{t^{i,k},\Y^{i,k}}(t^{{\tind}+1})\big)-\E f\big(\X^{t^{i,k+1},\Y^{i,k+1}}(t^{{\tind}+1})\big)\right)\\
&+\sum_{i=0}^{{\tind}-1}\left(\E f\big(\X^{t^{i,K},\Y^{i,K}}(t^{{\tind}+1})\big)-\E f\big(\X^{t^{i+1,0},\Y^{i+1,0}}(t^{{\tind}+1})\big)\right)
\\
&+\E f\big(\X^{t^{{\tind},K},\Y^{{\tind},K}}(t^{{\tind}+1})\big)-\E f\big(\Y^{{\tind}+1,0}\big).
\end{align*}
Making use of $\X^{t^{i,k},\Y^{i,k}}(t^{{\tind}+1})=\X^{t^{i,k+1},\X^{t^{i,k},\Y^{i,k}}(t^{i,k+1})}(t^{{\tind}+1})$ and combining the last two summands we obtain
\begin{align*}
&{\E f\big(\X^{t^0,X_0}(t^{{\tind}+1})\big)-\E f\big(\Y^{{\tind}+1,0}\big)}\\
=&\sum_{i=0}^{{\tind}}\sum_{k=0}^{K-1}
\left(\E f\big(\X^{t^{i,k+1},\X^{t^{i,k},\Y^{i,k}}(t^{i,k+1})}(t^{{\tind}+1})\big)-\E f\big(\X^{t^{i,k+1},\Y^{i,k+1}}(t^{{\tind}+1})\big)\right)\\
&+\sum_{i=0}^{{\tind}}\left(\E f\big(\X^{t^{i,K},\Y^{i,K}}(t^{{\tind}+1})\big)-\E f\big(\X^{t^{i+1,0},\Y^{i+1,0}}(t^{{\tind}+1})\big)\right).
\end{align*}
Using $\X^{t^{i,K},\Y^{i,K}}(t^{{\tind}+1})=\X^{t^{i+1,0},\X^{t^{i,K},\Y^{i,K}}(t^{i+1,0})}(t^{{\tind}+1})$ and the definition of $g$ this implies
\begin{align*}
&{\E f\big(\X^{t^0,X_0}(t^{{\tind}+1})\big)-\E f\big(\Y^{{\tind}+1,0}\big)}\\
=&\sum_{i=0}^{{\tind}}\sum_{k=0}^{K-1}
\left(\E g\big(t^{i,k+1},\X^{t^{i,k},\Y^{i,k}}(t^{i,k+1})\big)-\E g\big({t^{i,k+1},\Y^{i,k+1}}\big)\right)\\
&+\sum_{i=0}^{{\tind}}\left(\E f\big(\X^{t^{i+1,0},\X^{t^{i,K},\Y^{i,K}}(t^{i+1,0})}(t^{{\tind}+1})\big)-\E f\big(\X^{t^{i+1,0},\Y^{i+1,0}}(t^{{\tind}+1})\big)\right)
\\
=&\sum_{i=0}^{{\tind}}\sum_{k=0}^{K-1}
\left(\E g\big(t^{i,k+1},\X^{t^{i,k},\Y^{i,k}}(t^{i,k+1})\big)-\E g\big({t^{i,k+1},\Y^{i,k+1}})\big)\right)\\
&+\sum_{i=0}^{{\tind}}\left(\E g\big(t^{i+1,0},\X^{t^{i,K},\Y^{i,K}}(t^{i+1,0})\big)-\E g\big({t^{i+1,0},\Y^{i+1,0}}\big)\right).
\end{align*}
Thus, \eqref{eq:Konsuuniform} and \eqref{eq:KonsuEnsembleGen} imply
\begin{align*}
&|\E f\big(\X^{t^0,X_0}(t^{{\tind}+1})\big)-\E f\big(\Y_{[L]}(t^{{\tind}+1})\big)|
\\
\leq&\sum_{i=0}^{{\tind}}\sum_{k=0}^{K-1}\E C_g(\Y_{[L]}^{i,k})\dt^{\pos+1}
+\sum_{i=0}^{{\tind}}\left(
C_L\Dt+\tC_L(\Dt)^{\pe}\dt^{\pos}+\tC_L\dt^{\pos+1}
\right),
\end{align*}
which yields together with condition (\ref{St-lg-cond2}) the desired convergence.
\end{proof}

\subsection{A result using multistep state extrapolation\label{sec:conv-msem}}

For multistep state extrapolation, the analysis is complicated by the multistep nature of the method.
In this subsection we therefore restrict ourselves to consider linear SDEs \eqref{eq:linSDE}
with normally distributed initial values and restriction operator
\begin{equation}\label{eq:restrictideal}
\overline{\Q}(Y)=\begin{pmatrix}
\E Y\\
\Var Y
\end{pmatrix}.
\end{equation}
If we assume now equidistant coarse time steps and that the extrapolation step is given by \eqref{eq:MSextrap}, then we obtain
\begin{align*}
\E \Y^{{\tind}+1,0}=\sum_{s=0}^{\pe}l_s(\beta)\E \Y^{{\tind}-s,K},\quad \Var \Y^{{\tind}+1,0}=\sum_{s=0}^{\pe}l_s(\beta)\Var \Y^{{\tind}-s,K}
\end{align*}
with $l_s$ and $\beta$ given in \eqref{eq:lsm} and \eqref{eq:beta}. Application of the one step method $\os$ to \eqref{eq:linSDE} yields
\begin{align}\label{eq:osmssexlinsde}
\Y^{{\tind}+1,K}=&\hRos(a_1,t^{{\tind}+1,K-1},\dt,\eta^{{\tind}+1,K-1})\Y^{{\tind}+1,K-1}
\nonumber\\
&+\hSos(a_1,a_2,b,t^{{\tind}+1,K-1},\dt,\eta^{{\tind}+1,K-1})\\\nonumber
=&\prod_{k=0}^{K-1}\hRos(a_1,t^{{\tind}+1,k},\dt,\eta^{{\tind}+1,k})\Y^{{\tind}+1,0}
\\&+\sum_{k=0}^{K-1}\hSos(a_1,a_2,b,t^{{\tind}+1,k},\dt,\eta^{{\tind}+1,k})\prod_{i=k}^{K-1}\hRos(a_1,t^{{\tind}+1,i},\dt,\eta^{{\tind}+1,i}),\nonumber
\end{align}
where $\hRos$ and $\hSos$ are functions depending on $\os$, similar to the stability function in the deterministic case, and $\eta^{{\tind}+1,k}$ are (vectors of) i.\,i.\,d.\ random variables used by $\os$. Assuming that $\hRos(a_1,t^{{\tind}+1,k},\dt,\eta^{{\tind},k})$ is independent of $\eta^{{\tind},k}$, which holds, e.\,g., for typical Runge-Kutta methods, we obtain the multistep formulas
\begin{align}\nonumber
\E \Y^{{\tind}+1,K}=&\prod_{k=0}^{K-1}\hRos(a_1,t^{{\tind}+1,k},\dt)\sum_{s=0}^{\pe}l_s(\beta)\E \Y^{{\tind}-s,K}
\\&+\sum_{k=0}^{K-1}\E\hSos(a_1,a_2,b,t^{{\tind}+1,k},\dt,\eta^{{\tind}+1,k})\prod_{i=k}^{K-1}\hRos(a_1,t^{{\tind}+1,i},\dt),\label{eq:msemlinsdeexp}\\
\Var \Y^{{\tind}+1,K}=&\prod_{k=0}^{K-1}\hRos(a_1,t^{{\tind}+1,k},\dt)^2\sum_{s=0}^{\pe}l_s(\beta)\Var \Y^{{\tind}-s,K}\nonumber
\\&+\sum_{k=0}^{K-1}\Var\hSos(a_1,a_2,b,t^{{\tind}+1,k},\dt,\eta^{{\tind}+1,k})\prod_{i=k}^{K-1}\hRos(a_1,t^{{\tind}+1,i},\dt)^2.\label{eq:msemlinsdevar}
\end{align}
As due to the consistency of $\os$ we have $\hRos(a_1,t^{{\tind}+1,k},0)=1$ and \[\E\hSos(a_1,a_2,b,t^{{\tind}+1,k},0,\eta^{{\tind}+1,k})=\Var\hSos(a_1,a_2,b,t^{{\tind}+1,k},0,\eta^{{\tind}+1,k})=0,\]
the corresponding characteristic polynomial is given for both equations by
\begin{equation*}
P(\xi;\beta,{\pe})=\xi^{{\pe}+1}-\sum_{s=0}^{\pe}l_s(\beta)\xi^s.
\end{equation*}
As in the deterministic case (theory of linear multistep methods) we then obtain the following theorem.
\begin{theorem}
Assume that all roots $\xi$ of $P(\xi;\beta,{\pe})=0$ lie within the unit circle and that all roots with absolute value one are simple. If then the one-step method is weakly consistent of order $\pos$ and given by \eqref{eq:osmssexlinsde} with $\Ros$ independent of $\eta$, then coarse multistep state extrapolation with restriction operator \eqref{eq:restrictideal} and extrapolation given by \eqref{eq:MSextrap} is  convergent of order $p=\min\{\pe,\pos\}$ for linear SDEs \eqref{eq:linSDE} with normally distributed initial values.
\end{theorem}

\section{Numerical results\label{sec:numerical}}

In this section, we provide further numerical results.  We first illustrate the dependence of the local error in one accelerated time step on the step size (\cref{sec:num-extrapol}). Subsequently, we perform a number of long-term simulations (\cref{sec:num-long}).

Below, we consider \cref{eq:fene-3d} in one space dimension, with  $F(X)$ the FENE force \eqref{eq:springs} with $\fp=49$, $\We=1$. As the velocity field, we choose $\kappa(t)= 2\cdot(1.1+\sin\left(\pi t)\right)$, and we again sample the initial states from the invariant distribution of \cref{eq:fene-3d} for $\kappa(t)\equiv 0$, and use $\epsilon=1$ in \eqref{eq:kramers}. We discretize in time with the classical Euler-Maruyama scheme with time step $\dt=2\cdot 10^{-4}$. As before, the macroscopic state $\U_{[L]}$ consists of the first $L$ \emph{even} centralized moments.  We study the micro/macro acceleration algorithm with first order projective extrapolation, and with first and second order multistep state extrapolation. In all cases, we perform $K=1$ microscopic steps before extrapolation.

\subsection{Local error\label{sec:num-extrapol}}
We simulate $\Eind=1\cdot 10^5$ realizations up to time $t^*=1.6$, and record the microscopic state $\bY^{-}$ at time $t^-=1.4$, as well as the  macroscopic states $\U_{[L]}(t^-+\Dt)$ for $L=3,4,5$, and the approximated function of interest $\tilde{\tau}_p(t^-+\Dt)$ for $\Dt \in [0,t^*-t^-]$. We then extrapolate from time $t^-$ to time $t=t^-+\Dt$, and project the ensemble $\bY^{-}$ onto  $\U_{[L]}(t^-+\Dt)$.   Subsequently, we compute the corresponding value of the stress as $\hat{\tau}_p(t^-+\Dt)$. We record the relative error with respect to the reference solution, $|\hat{\tau}_p(t)-\tilde{\tau}_p(t)|/\tilde{\tau}_p(t)$, as a function of $\Dt$.
To reduce the statistical error, we report the averaged results of $50$ realizations of this experiment. Then, the statistical error has an order of magnitude of about $10^{-5}$. The results are shown in \cref{fig:extrap-many}.
\begin{figure}
\begin{center}
	\includegraphics[width=0.7\linewidth]{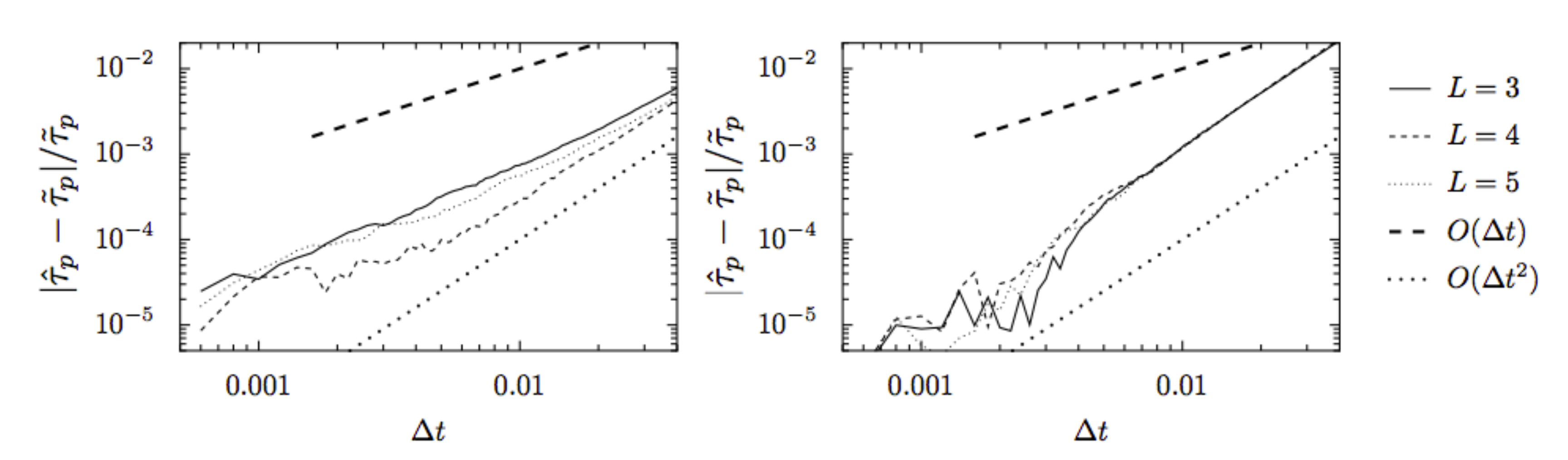}
	\end{center}
\caption{\label{fig:extrap-many}Error of the stress after extrapolating and projecting a prior ensemble of $J=1\cdot 10^5$ FENE dumbbells onto the first $L$ even centralized moments of a reference ensemble, as a function of $\Dt$ for several values of $L$. Displayed is the result averaged over $50$ realizations of the experiment. Left: First order projective extrapolation. Right: First order multistep state extrapolation. Simulation details are given in the text.}
\end{figure}

For projective integration, we clearly see a first order behavior as a function of $\Dt$; this is a consequence of the amplification of the statistical error during projective extrapolation.  Note that, due to the presence of three competing sources of errors (extrapolation, matching, and statistical error), which may be of opposite signs, the effect of extrapolating more macroscopic state variables is not so clearly visible as in \cref{fig:projection-Dt}. Note that for $L=4$ the statistical and matching errors seem to be a bit lower, such that the second order behavior of the extrapolation error is already apparent for the largest displayed time steps.

For multistep state extrapolation, the situation is slightly different.  Here, we see that, for modest gains (small $\Dt$), the statistical error remains more or less unaffected. For $\Dt>2\cdot 10^{-3}$ (gain factor $10$), however, the error increases as $\Dt^2$, as a consequence of the large extrapolation error.  Note that, as soon as the extrapolation error dominates, the error appears to be independent of the number of moments used.

To emphasize the effect of the statistical error, we repeat the experiment using $\Eind=1000$ realizations (averaged over $20$ realizations of the experiment). The results are shown in \cref{fig:extrap-less}.
\begin{figure}
\begin{center}
	\includegraphics[width=0.7\linewidth]{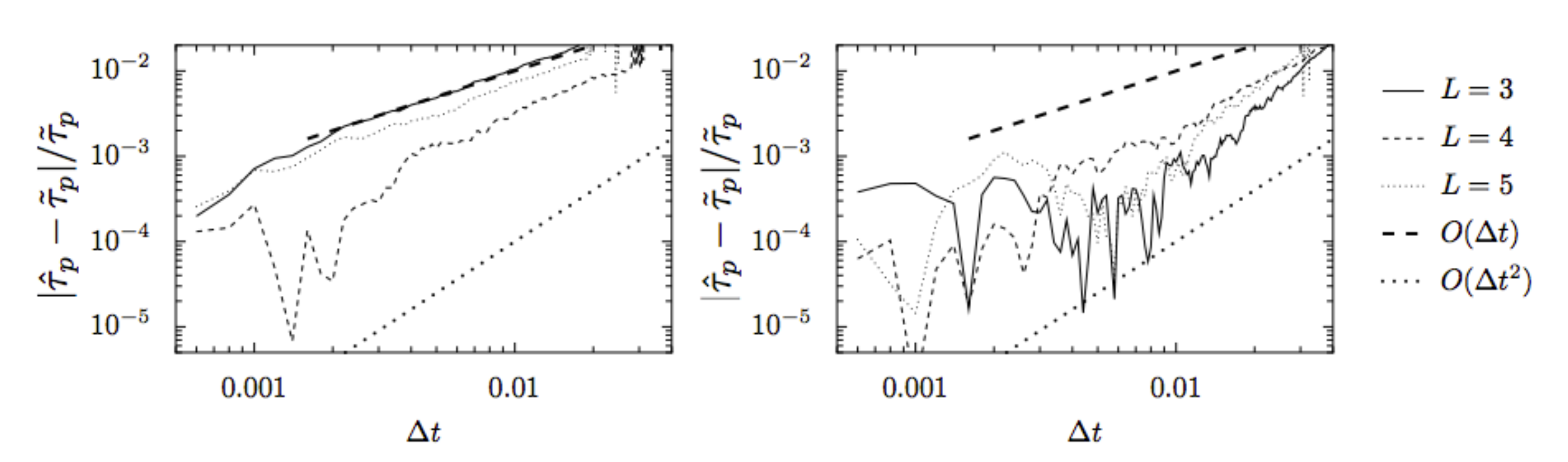}
	\end{center}
\caption{\label{fig:extrap-less} Error of the stress after extrapolating and projecting a prior ensemble of $J=1\cdot 10^3$ FENE dumbbells onto the first $L$ even centralized moments of a reference ensemble, as a function of $\Dt$ for several values of $L$. Displayed is the result averaged over $20$ realizations of the experiment. Left: First order projective extrapolation. Right: First order multistep state extrapolation. Simulation details are given in the text.}
\end{figure}
Compared to \cref{fig:extrap-many}, we see qualitatively the same behavior. Again, the error of projective extrapolation increases linearly (but it is now an order of magnitude larger), while, for multistep state extrapolation, larger gains appear to be possible since the statistical error now dominates for a wider range of extrapolation step sizes $\Dt$.

\subsection{Long-term simulation\label{sec:num-long}}

We now turn to a long-term simulation,  and compare the behavior of the sample mean and sample standard deviation of a full microscopic simulation (which we will call the reference simulation) with the micro/macro acceleration algorithm.
We denote by $\tilde{\tau}_p(t)$ the approximation to the function of interest calculated from one realization of the reference simulation using $J$ SDE realizations, and by $\hat{\tau}_p(t)$ the function of interest obtained via one realization of the micro/macro acceleration technique.
As extrapolation techniques, we use first order projective extrapolation and second order multistep state extrapolation. (For the set-up in this example, the deterministic error of first order multistep state extrapolation is too high to be considered further.)

\Cref{fig:stat-fene-proj-mom} (top left) shows the evolution of the stress as a function of time.  We see that the simulation exhibits a periodic behavior, with a fast increase of the stress followed by a relaxation.  Note that, in this problem, due to the fast variations in $\kappa(t)$, there is not a very strong time-scale separation between the evolution of the stress tensor and the evolution of individual polymers.  Hence, using the time-step adaptation strategy (see \Cref{sec:proj-failure}) will prove crucial to obtain an efficient algorithm.   During the fast increase of the stress, we observed that the matching operator fails for time steps $\Delta t>1\cdot 10^{-3}$, whereas larger accelerations are possible during the relaxation. Therefore, in this experiment we will use adaptive macroscopic time steps, as outlined in \cref{sec:proj-failure}, choosing $\underline{\alpha}=0.2$ and $\overline{\alpha}=1.2$.  On average, we obtained a speed-up factor of $4$, meaning that microscopic simulation has been performed over $1/4$ of the time domain.

In a first experiment, we use  $\Delta t=1\cdot 10^{-3}$ and vary the number $L$ of macroscopic state variables.  The results for 500 realizations of this experiment, each with an ensemble size of $J=5000$, are shown in \cref{fig:stat-fene-proj-mom}.

\begin{figure}
\begin{center}
	\includegraphics[width=\linewidth]{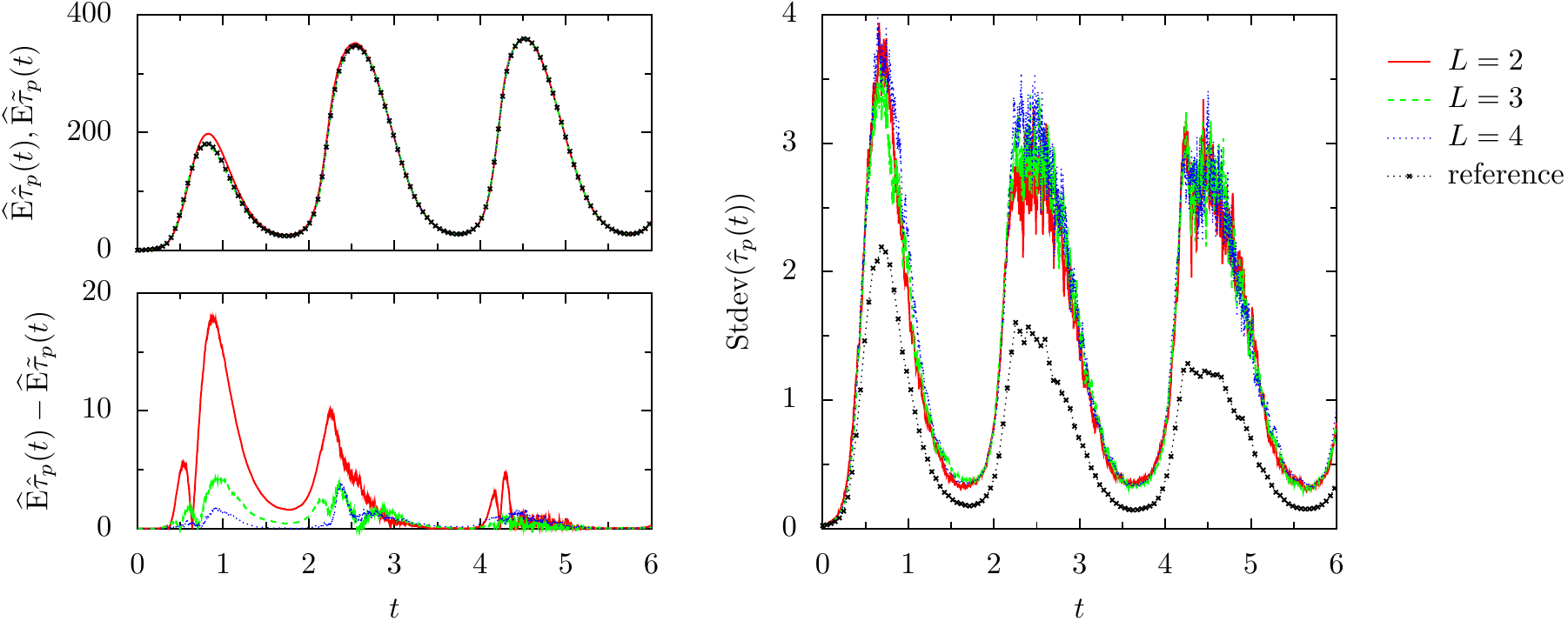}
	\end{center}
\caption{\label{fig:stat-fene-proj-mom}Results of micro/macro acceleration of the FENE model \eqref{eq:fene-3d} using $\Delta t=1\cdot 10^{-3}$ and projective extrapolation for different numbers $L$ of macroscopic state variables, as well as a full microscopic (reference) simulation.  Top left: evolution of the sample means of the stresses $\tilde{\tau}_p$ and $\hat{\tau}_p$. Bottom left: deterministic error on $\hat{\tau}_p$. Right: evolution of the sample standard deviation of $\hat{\tau}_p$. Simulation details are given in the text.}
\end{figure}
We make two main observations.  First, the deterministic error decreases with increasing $L$, whereas the sample standard deviation is independent of $L$. (The different lines in the plot are nearly indistinguishable.) Note also that the variance on the sample standard deviation is quite large in this example. Second, from \cref{fig:stat-fene-proj-mom}, we see that the error of the micro/macro acceleration algorithm with respect to the reference simulation also decreases as a function of time, until it reaches a level of the order of the statistical error.  This behavior can be attributed to the fact that, in this example, the macroscopic behavior of the system on long time scales is determined by only a few macroscopic state variables. The results for multistep state extrapolation (not shown) in terms of $L$ are similar.

In a second experiment, we fix $L=3$ and consider varying $\Delta t$.  (This experiment and its conclusions closely resemble the one in \cref{sec:num-extrap}.) The results for $500$ realizations, each with an ensemble size of $J=1000$, are shown in \cref{fig:stat-fene-proj}.
\begin{figure}
\begin{center}
	\includegraphics[width=\linewidth]{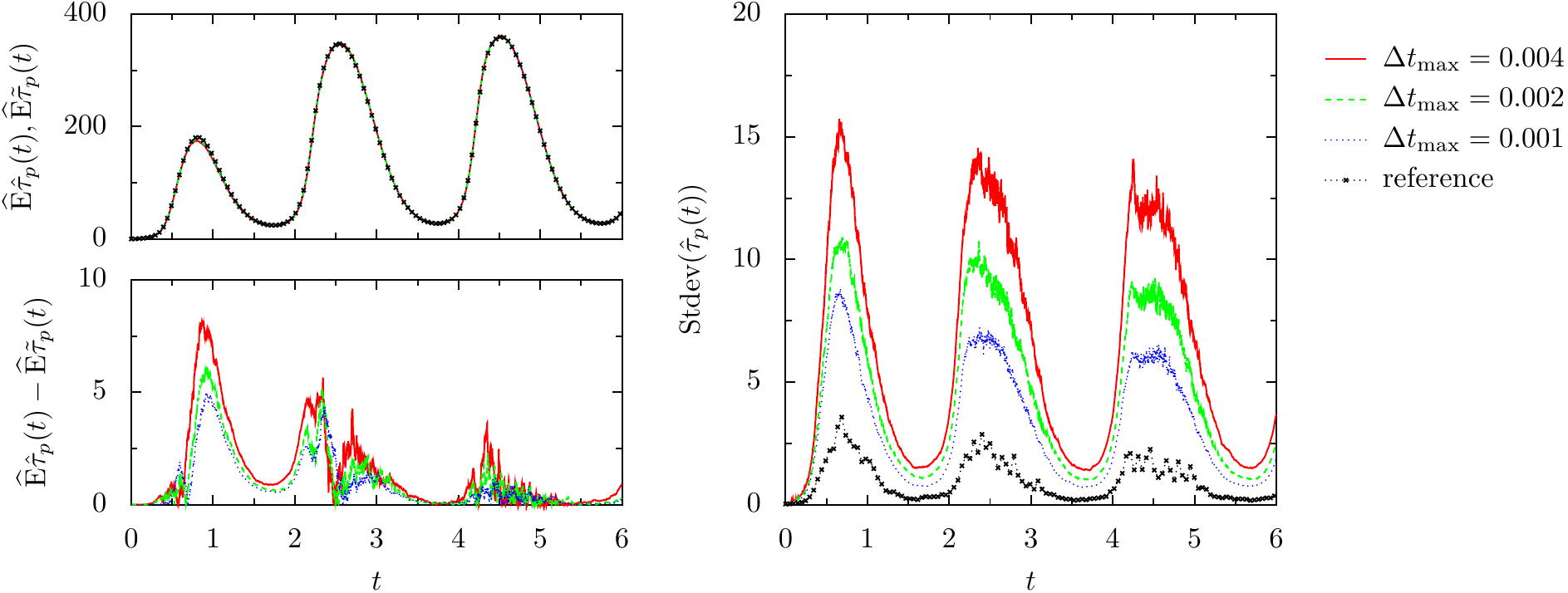}
	\end{center}
\caption{\label{fig:stat-fene-proj}Results of micro/macro acceleration of the FENE model \eqref{eq:fene-3d} using $L=3$ moments and projective extrapolation for different values of $\Delta t$, as well as a full microscopic (reference) simulation.   Top left: evolution of the sample means of the stresses $\tilde{\tau}_p$ and $\hat{\tau}_p$. Bottom left: deterministic error on $\hat{\tau}_p$. Right: evolution of the sample standard deviation of $\hat{\tau}_p$. Simulation details are given in the text.}
\end{figure}
\Cref{fig:stat-fene-proj} (bottom left) shows again that the deterministic error grows with increasing $\Delta t$, whereas the right figure illustrates that the sample standard deviation is larger for larger $\Delta t$. For second order multistep state extrapolation, we obtain \cref{fig:stat-fene-msem-2}.
\begin{figure}
\begin{center}
	\includegraphics[width=\linewidth]{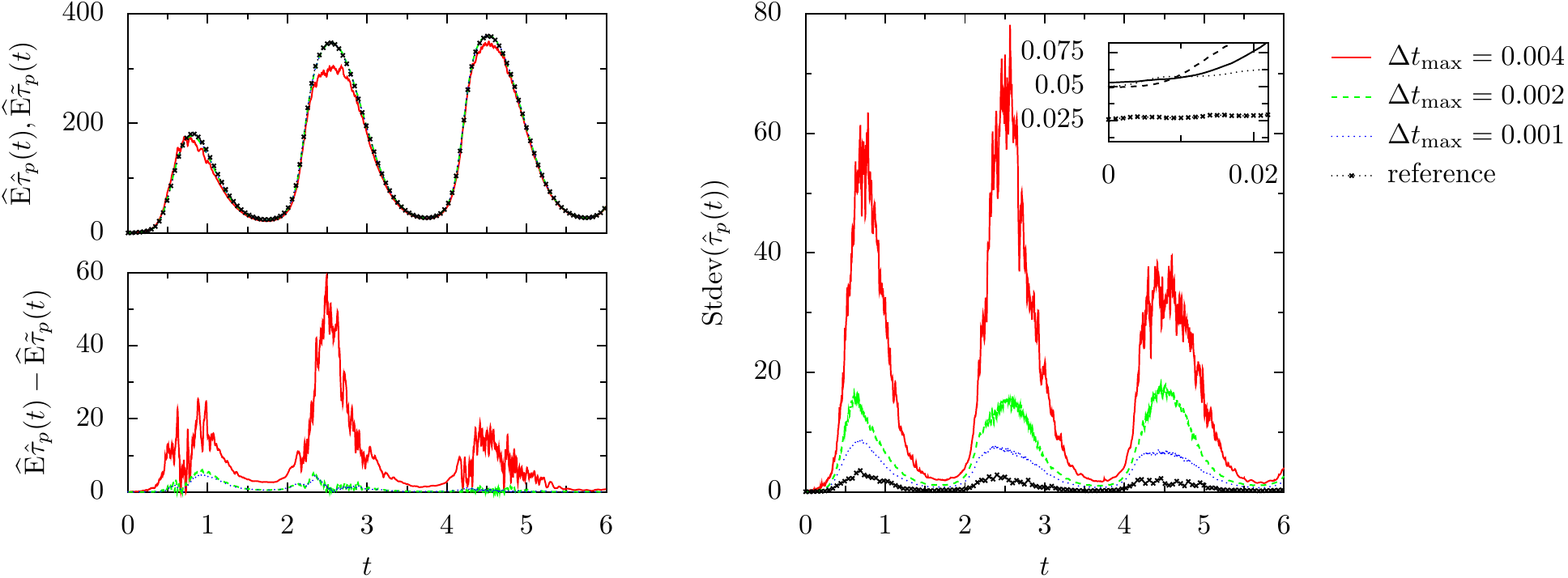}
	\end{center}
\caption{\label{fig:stat-fene-msem-2}Results of micro/macro acceleration of the FENE model \eqref{eq:fene-3d} using $L=3$ moments and second order multistep state extrapolation for different numbers $L$ of macroscopic state variables, as well as a full microscopic (reference) simulation.   Top left: evolution of the sample means of the stresses $\tilde{\tau}_p$ and $\hat{\tau}_p$. Bottom left: deterministic error on $\hat{\tau}_p$. Right: evolution of the sample standard deviation of $\hat{\tau}_p$. Simulation details are given in the text.}
\end{figure}
The behavior of the  sample standard deviation  is similar to the linear case. When zooming in to the behavior for small $t$, we observe that, over a short time interval, the  sample standard deviation  using the micro/macro acceleration technique increases at the same rate as the  sample standard deviation  in the reference simulation, which corresponds to the theoretical result on the local propagation of statistical error. However, on longer time scales, we again see that the  sample standard deviation for large $t$ grows rapidly, and seems to be larger for larger $\Delta t$, due to accumulation effects. Note that, in this case, second order multistep state extrapolation even behaves worse than first order projective integration on long time scales for sufficiently large $\Delta t$.

\section{Conclusions and outlook\label{sec:concl}}

We presented and analyzed a micro/macro acceleration technique for the Monte Carlo simulation of stochastic differential equations (SDEs) in which short bursts of simulation using an ensemble of microscopic SDE realizations are combined with an extrapolation of an estimated macroscopic state forward in time.  The method is designed for problems in which the required time step for each realization of the SDE is small compared to the time scales on which the function of interest evolves.  For such systems, one often needs to take a very small microscopic time step, which results in a deterministic error that is much smaller than the statistical error.

We showed that the proposed procedure converges in the absence of statistical error, provided the matching operator satisfies a number of natural conditions, and we introduced a matching operator that satisfies these conditions for Gaussian random variables.  We also conjectured that this matching operator is suitable for general distributions, and provided numerical evidence to support this conjecture.
Concerning the statistical error, a local analysis of projective extrapolation shows that the amplification of statistical error depends on the ratio $\alpha$ of macroscopic (extrapolation) and microscopic (simulation) time steps, while this is not the case for multistep state extrapolation. Numerical evidence, however, suggests that, when using higher order multistep state extrapolation, accumulation of statistical error over macroscopic time scales may nevertheless induce an $\alpha$-dependent statistical error.

This paper has not focused on quantifying the computational gains that can be expected from this method.  It is clear from the description that the method will be more efficient when there is a bigger separation in time scales between the microscopic and macroscopic levels.  The numerical examples in this text do not exhibit such a strong time-scale separation; they were mainly chosen for their ability to clearly illustrate the effects of the different sources of numerical error.  However, some conclusions on efficiency can nevertheless be drawn. For a given required variance on the solution, the computational cost using first order projective extrapolation is comparable to that of a full simulation, since the former requires more SDE realizations due to the $\alpha$-dependent amplification of statistical error. For first order multistep state extrapolation, extrapolation without such a drastic amplification of statistical error is possible, at the cost of an amplified deterministic error. This can be acceptable if the macroscopic function of interest changes slowly compared to the time step of the SDE.

We note that, for the model problem of coupled micro/macro simulation of dilute polymer solutions, the amplification of statistical error using projective extrapolation need not be dramatic: while the computational cost of a micro/macro accelerated simulation and a full microscopic simulation are comparable for given variance, this is no longer true when coupling this Monte Carlo simulation to a PDE for the solvent.  Indeed, when extrapolating the complete coupled system forward in time, a computational gain is obtained since the PDE for the solvent also does not need to be simulated on the whole time domain.

In future work, we will study stability and propagation of statistical error on long time scales. The numerical experiments indicate that these issues can be studied in a linear setting. Another open question is for which distributions the conjecture can be proved.  From an algorithmic point of view, this work raises questions on the adaptive/automatic selection of all method parameters (number of moments to extrapolate, macroscopic time step, number of SDE realizations) to ensure a reliable computation with minimal computational cost.  Also, a numerical comparison with other approaches, such as implicit approximations, could be envisaged.

\end{document}